\documentclass[reqno,11pt]{amsart}
\usepackage{amsmath,amsfonts,amssymb,amsthm}\usepackage[foot]{amsaddr}
\usepackage{mathtools,mathrsfs,yhmath}
\usepackage{graphicx,color,xcolor}
\usepackage{enumitem}
\usepackage{hyperref}

\usepackage{blindtext}
\usepackage{multicol}
\usepackage{color}
\usepackage{comment}
\usepackage{wrapfig}
\usepackage{graphicx}

\usepackage{hyperref}
\hypersetup{colorlinks=true, linkcolor=blue, filecolor=magenta, urlcolor=blue}
\usepackage[capitalize,nameinlink]{cleveref}
\usepackage[backend=biber, sorting=nty]{biblatex}
\bibliography{Vlasov.bib}

 \allowdisplaybreaks
\numberwithin{equation}{section}
\crefname{enumi}{}{parts}

\newcommand{\B}{\mathrm{B}}
\newcommand{\T}{\mathrm{T}}
\newcommand{\R}{\mathbb{R}}

\newcommand{\D}{\mathcal{D}}

\newcommand{\HW}{\mathrm{HW}}

\newcommand{\X}{\mathcal{X}}
\newcommand{\Rd}{\R^d}
\newcommand{\Td}{\mathbb{T}^d}

\newcommand{\RdRd}{\R^d \times \R^d}
\newcommand{\XRd}{\X \times \R^d}
\DeclarePairedDelimiter{\abs}{\lvert}{\rvert}
\DeclarePairedDelimiter{\norm}{\lVert}{\rVert}

\DeclareMathOperator{\Id}{Id}

\numberwithin{equation}{section}
\newtheorem{theorem}{Theorem}[section]
\newtheorem{lemma}[theorem]{Lemma}
\newtheorem{proposition}[theorem]{Proposition}

\newtheorem{definition}[theorem]{Definition}
\newtheorem{assumption}[theorem]{Assumption}
\crefname{assumption}{Assumption}{Assumptions}

\newtheorem{remark}[theorem]{Remark}
\newcommand\numberthis{\addtocounter{equation}{1}\tag{\theequation}}

\makeatletter
\let\oldabs\abs
\def\abs{\@ifstar{\oldabs}{\oldabs*}}
\let\oldnorm\norm
\def\norm{\@ifstar{\oldnorm}{\oldnorm*}}
\makeatother

\begin{document}

\title[Stability Estimates for the Vlasov--Poisson System with Yudovich Density]{Stability Estimates in Kinetic Wasserstein Distances for the Vlasov--Poisson System with Yudovich Density}

\author{Jonathan Junné}
\address{TU Delft, Delft Institute of Applied Mathematics, Mekelweg 4, 2628 CD Delft, Netherlands}
\email{\href{mailto:j.junne@tudelft.nl}{j.junne@tudelft.nl}}
\author{Alexandre Rege}
\address{ETH Z\"urich, Department of Mathematics, R\"amistrasse 101, 8092 Z\"urich, Switzerland}
\email{\href{mailto:alexandre.rege@math.ethz.ch}{alexandre.rege@math.ethz.ch}}

\subjclass[2020]{35Q83, 82C40, 82D10, 35B35}
\keywords{Vlasov--Poisson, kinetic Wasserstein distances, Yudovich spaces, stability estimates}
\begin{abstract}
    We investigate the stability of solutions to the Vlasov--Poisson system using the unifying framework of the kinetic Wasserstein distance, introduced by Iacobelli in (Section 4 in \textit{Arch. Ration. Mech. Anal.} \textbf{244}  (2022), no. 1, 27-50). This allows us to treat both macroscopic densities that lie in a Yudovich space, as recently considered by Crippa et al. (Theorem 1.6 in \textit{Nonlinearity} \textbf{37} (2024), no. 9, 095015) for the $1$-Wasserstein distance, and higher order Wasserstein distances, for which only bounded macroscopic densities were treated by Iacobelli and the first author (Theorem 1.11 in \textit{Bull. Lond. Math. Soc.} \textbf{56} (2024), 2250-2267). First, we establish an $L^p$-estimate on the difference between two force fields in terms of a suitable nonlinear quantity that controls the kinetic Wasserstein distance between their macroscopic densities. Second, we use this estimate in order to derive a closable Osgood-type inequality for the kinetic Wasserstein distance between two solutions. 
    This enables us to prove our main theorem; for $1 \le p < +\infty$ we show the $p$-Wasserstein stability of solutions to the Vlasov--Poisson system with macroscopic densities belonging to a Yudovich space.
\end{abstract}

\maketitle

\section{Introduction}\label{Sec:Introduction}
\subsection{The Vlasov--Poisson system}\label{sec:The Vlasov--Poisson system}
We consider the Vlasov--Poisson system for electrostatic or gravitational interaction. This system describes the dynamics of interacting particles, which are modelled by a distribution function $f \equiv f(t,x,v)$ depending on time $t \in \R^+$, position $x \in \X$ denoting either the $d$-dimensional torus $\Td$ or the Euclidean space $\Rd$, and velocity $v \in \R^d$. It is given by the following set of equations:
\begin{equation}\label{sys:Vlasov--Poisson}\tag{VP}
    \begin{cases}
        \partial_t f + v\cdot \nabla_x f - \nabla U_f \cdot \nabla_{v} f = 0, \\
        \sigma \Delta U_f = \rho_f - 1_{\X = \Td},\\
        \rho_f := \int_{\Rd} f \: dv,
    \end{cases}
\end{equation}
where $\sigma = \pm 1$ encodes either gravitational ($\sigma=1$) or electrostatic ($\sigma=-1$) interaction and $U_f \equiv U_f(t;x)$ is the self-induced gravitational or electrostatic potential.

Throughout this paper, we will consider two Lagrangian weak solutions $f_1$ and $f_2$ to \eqref{sys:Vlasov--Poisson} with respective initial data $f_1(0), f_2(0)$, which are probability distributions, and respective flows $Z_1 := (X_1, V_1)$ and $Z_2 := (X_2, V_2)$. These flows satisfy the system of characteristics 
\begin{equation*}\label{sys:Characteristics}\tag{CH-VP}
    \dot{X} = V, \quad \dot{V} = -\nabla U_f, \quad X(0; x, v) = x, \quad V(0; x, v) = v.
\end{equation*}
Lagrangian weak solutions are weak solutions that can be written as pushforwards of the initial data by flows solving \eqref{sys:Characteristics}; in this case $f_1(t) = Z_1(t;  \cdot, \cdot)_\# f_1(0)$ and $f_2(t) = Z_2(t; \cdot, \cdot)_\# f_2(0)$. For notational convenience, unless otherwise stated, all quantities are always evaluated at time $t$ and we denote $X_1 := X_1(t;x,v), \,V_1 := V_1(t;x,v),\, X_2 := X_2(t;y,w),\, V_2 := V_2(t;y,w),\, \nabla U_{f_i}(X_j) := \nabla U_{f_i}(t; X_j)$ for $i,j = 1,2$. 

We investigate the stability of solutions to \eqref{sys:Vlasov--Poisson} in the Wasserstein distances of order $p \geq 1$, which we recall on the phase space $\X\times\Rd$ for convenience (see \cite[Chapter 6]{villani_optimal_2009}):
\begin{definition}[Wasserstein distances]\label{def:Wasserstein}
    Let $\mu, \nu$ be two probability measures on $\XRd$. The \emph{Wasserstein distance} of order $p \ge 1$, between $\mu$ and $\nu$ is defined as
    \begin{equation*}
        W_p(\mu, \nu) := \left( \inf_{\pi \in \Pi(\mu, \nu)} \left\{\int_{(\XRd)^2} \abs{x - y}^p + \abs{v - w}^p \: d\pi(x,v,y,w)\right\} \right)^{\frac{1}{p}},
    \end{equation*}
    where $\Pi(\mu, \nu)$ is the set of \emph{couplings}; that is, the set of probability measures with \emph{marginals} $\mu$ and $\nu$. A coupling is said to be \emph{optimal} if it minimizes the Wasserstein distance.
\end{definition}

We begin by mentioning the major contribution of Loeper \cite[Theorem 1.2]{loeper_uniqueness_2006}, where the uniqueness and stability of the Vlasov–Poisson system \eqref{sys:Vlasov--Poisson}, assuming the macroscopic density remains bounded, were established using an optimal transport argument involving the 2-Wasserstein distance. Loeper's main observation was that the $L^2$-norm of the difference of force fields can be controlled in terms of the 2-Wasserstein distance between their macroscopic densities \cite[Theorem 2.9]{loeper_uniqueness_2006}. 

The uniqueness criterion was subsequently improved by Miot \cite[Theorem 1.1]{miot_uniqueness_2016} relying on the second order structure of the characteristic ODE \eqref{sys:Characteristics} to allow for macroscopic densities satisfying
\begin{equation*}
    \max_{t \in [0,T]}\sup_{r\ge 1} \frac{\norm{\rho_f(t)}_{L^r(\Rd)}}{r} < +\infty,
\end{equation*}
and succeeded in highlighting unbounded densities that satisfy this condition \cite[Theorem 1.2]{miot_uniqueness_2016}.
These results were later generalized, together with Holding \cite[Theorem  1.1]{holding_uniqueness_2018}, to provide an effective $W_1$ stability estimate allowing for exponential Orlicz macroscopic densities verifying
\begin{equation*}
    \max_{t \in [0,T]}\sup_{r\ge \alpha >1} \frac{\norm{\rho_f(t)}_{L^r(\Rd)}}{r^{1/\alpha}} < +\infty.
\end{equation*}

Finally, in a very recent contribution \cite{crippa_existence_2024}, Crippa, Inversi, Saffirio and Stefani established a $W_1$ stability estimate for macroscopic densities within uniformly-localized Yudovich spaces \cite[Theorem 1.6]{crippa_existence_2024}. Their Lagrangian approach was inspired by a previous work of Crippa and Stefani \cite{crippa_elementary_2024} on the well-posedness of the two-dimensional Euler equations. Similarly to Miot's work, by using the propagation of moments available for \eqref{sys:Vlasov--Poisson}, they provided explicit examples of unbounded macroscopic densities in this uniformly localized Yudovich space \cite[Theorem 1.7 \& Proposition 1.8]{crippa_existence_2024}. 

Compared to Loeper's approach, the aforementioned $W_1$ stability estimates not only accommodate a broader class of macroscopic densities for which the uniqueness of \eqref{sys:Vlasov--Poisson} is guaranteed, but also provide improved control over the time interval during which two initially close solutions remain close. Generally, the $1$-Wasserstein distance is more flexible, and, the $W_1$ stability estimates obtained so far do not rely on optimal transport theorems, whereas Loeper's stability estimate in the stronger $2$-Wasserstein distance heavily depends on them. 

To improve control over the time interval in Loeper's stability estimate and obtain an effective $W_2$ stability estimate \cite[Theorem 3.1]{iacobelli_new_2022}, Iacobelli introduced the kinetic Wasserstein distance in \cite{iacobelli_new_2022}, specifically designed to account for the anisotropy between position and momentum that is inherent to kinetic problems (see \cite[Section 4]{iacobelli_new_2022}):
\begin{definition}[kinetic Wasserstein distance]\label{def:kinetic Wasserstein distance}
    Let $\mu, \nu$ be two probability measures on $\XRd$. The \emph{kinetic Wasserstein distance} of order $p$, with $p \ge 1$, between $\mu$ and $\nu$ is defined as
    \begin{equation*}
        W_{\lambda, p}(\mu, \nu) := \left(\inf_{\pi \in \Pi(\mu, \nu)} \left\{D_p(\pi, \lambda)\right\}\right)^{\frac{1}{p}},
    \end{equation*}
    where $D_p(\pi, \lambda)$ is the unique number $s$ such that
    \begin{equation*}
        s - \lambda(s)\int_{(\X \times \Rd)^2} \abs{x - y}^p \: d\pi(x, v, y, w) = \int_{(\X \times \Rd)^2} \abs{v - w}^p \: d\pi(x, v, y, w),
    \end{equation*}
    with $\lambda : \R^+ \to \R^+$ a decreasing function.
\end{definition}
The kinetic Wasserstein distance allowed Iacobelli to obtain the quasi-neutral limit of \eqref{sys:Vlasov--Poisson} towards the kinetic isothermal Euler system allowing for $W_2$ perturbations of analytic data with almost optimal rate \cite[Corollary 3.3 \& Remark 3.4]{iacobelli_new_2022}, as a consequence of the enhanced $W_2$ stability estimate. Subsequently, with Griffin-Pickering, they obtained the stability of the quasi-neutral limit for the ionic  with exponentially small Wasserstein perturbations \cite[Theorem 1.1]{griffin-pickering_stability_2024} (see \cite{griffin-pickering_recent_2021} for a recent account on the subject). The kinetic Wasserstein distance approach also allowed Iacobelli and Lafleche \cite[Theorem 2.4]{iacobelli_enhanced_2024} to improve the convergence rate in the semi-classical analogue to the $2$-Wasserstein distance of the semi-classical limit of the quantum Hartree system towards the classical \eqref{sys:Vlasov--Poisson}.

Regarding improved stability estimates, we also mention a recent work of Iacobelli and the first author \cite[Theorems 1.9 \& 1.11]{iacobelli_stability_2024} where Loeper's approach and the kinetic Wasserstein approach were generalized to the $p$-Wasserstein setting for $1 < p < +\infty$  and of the second author \cite[Theorem 2.1]{rege_stability_2024} dealing with the stability of solutions for the magnetized Vlasov--Poisson system \eqref{sys:Vlasov--Poisson weak B}, where the kinetic Wasserstein distance was necessary to control the additional terms coming from the varying external magnetic field. In both works, the controls on the time interval match the effective stability estimate in the weaker $1$-Wasserstein distance obtained by Holding and Miot \cite{holding_uniqueness_2018, iacobelli_new_2022}.

\subsection{Yudovich spaces and assumptions}\label{sec:Yudovich spaces and assumptions}
The two macroscopic densities $\rho_{f_1}$ and $\rho_{f_2}$ are assumed to be absolutely continuous with respect to the Lebesgue measure and to belong to a Yudovich space (see \cite[Section 1.2]{crippa_existence_2024}):
\begin{definition}\label{def:Yudovich space}
    Given a nondecreasing function $\Theta : [0, +\infty) \to (0, +\infty)$, called a \emph{growth function}, the \emph{Yudovich space} $Y^\Theta(\X)$ consists of all functions $h$ that belong to the intersected Lebesgue spaces $\cap_{1 \le r < +\infty} L^r(\X)$ such that
    \begin{equation*}
        \norm{h}_{Y^\Theta(\X)} := \sup_{1 \le r < +\infty} \frac{\norm{h}_{L^r(\X)}}{\Theta(r)} < +\infty.
    \end{equation*}
\end{definition}
Note that only the asymptotic behavior of $\Theta$ matters in the definition of the Yudovich space; for convenience, assume $\Theta(1) = 1$. Throughout the text, we denote by $\cdot'$ the Lebesgue conjugate exponent.
\begin{assumption}\label{ass:L^1-bound}
    For some $T > 0$, we assume that
\begin{equation*}
    A(t) := \max\left\{\norm{\rho_{f_1}(t)}_{Y^\Theta(\X)}, \norm{\rho_{f_2}(t)}_{Y^\Theta(\X)}\right\}^{1 + \frac{1}{p'}}\left(1_{p=1} + 1_{1<p<+\infty}\norm{\rho_{f_2}}_{Y^\Theta(\X)}^{\frac{1}{p}}\right) \in L^1([0, T]).
\end{equation*}
\end{assumption}
We consider suitable growth function for \eqref{sys:Vlasov--Poisson} such that the Cauchy problem \eqref{sys:Characteristics} is well-posed:
\begin{assumption}\label{ass:varphi_theta-continuous}
    The growth functions $\Theta$ is such that the associated generalized modulus of continuity $\varphi_\Theta : [0, +\infty) \to [0, +\infty)$ is continuous, where
    \begin{equation*}
        \varphi_\Theta(s) := \begin{cases}
            0 & \mbox{if } s=0, \\
            s\abs{\log s}\Theta(\abs{\log s}) &\text{if } 0 < s < e^{-d-1}, \\
            e^{-d-1}(d+1)\Theta(d+1) &\text{if } s \ge e^{-d-1}.
        \end{cases}
    \end{equation*}
\end{assumption}
Indeed, $\varphi_\Theta$ is the modulus of continuity of the force field given that the macroscopic density is Yudovich, which is enforced by \cref{ass:L^1-bound} (see Crippa, Inversi, Saffirio and Stefani \cite[Lemma 1.1 and Assumption 1.3]{crippa_existence_2024}); as explained in \cite{crippa_existence_2024}, the value $e^{-{d-1}}$ in the definition of $\varphi_\Theta$ is essentially irrelevant and included solely to make $\varphi_\Theta$ more appealing. Under \cref{ass:L^1-bound} and \cref{ass:varphi_theta-continuous}, one can define weak solutions $f$ to \eqref{sys:Vlasov--Poisson} through
\begin{equation*}
    \int_0^T\int_{\XRd} \left[\left(\partial_t \phi + v \cdot \nabla_x \phi - \nabla U_f \cdot \nabla_v \phi\right)f\right]\left(t;x,v\right) \: dxdv\:dt = -\int_{\XRd} \phi(0, x) f(0;x,v) \: dxdv
\end{equation*}
for all test functions $\phi \in C_c^\infty([0, T) \times (\XRd))$, since then the product of the solution with the force field is integrable; i.e., $\norm{f(t)\nabla U_f(t)}_{L^1(\XRd)} \in L^1([0, T])$.

While \cite[Theorem 1.6]{crippa_existence_2024} assumed $\varphi_{\Theta}$ to be nondecreasing concave in some regime for the $1$-Wasserstein stability of \eqref{sys:Vlasov--Poisson}, in the $p$-Wasserstein setting, we instead assume a $p$-modified version of $\varphi_\Theta$ to be nondecreasing concave in some regime as follows:
\begin{assumption}\label{ass:Theta-non-drecreasing-and-concave}
    The growth function $\Theta$ is such that $\varphi_{p,\Theta} : [0, +\infty) \to [0, +\infty)$ is nondecreasing concave on $[0, c_{p,\Theta;d})$ for some positive constant $c_{p,\Theta;d} < 1/e$ that depends only on $p$, $\Theta$, and $d$, where $\varphi_{p;\Theta}$ is given by
    \begin{equation*}
        \varphi_{p,\Theta}(s) := \begin{cases}
            0 &\text{if } s=0,  \\
            s\abs{\log s}^p\Theta^p\left(\abs{\log s}\right) &\text{if } 0 < s \le c_{p,\Theta;d}, \\
            \varphi_{p,\Theta}(c_{p,\Theta;d}) &\text{if } s \ge c_{p,\Theta;d}.
        \end{cases}
    \end{equation*}
\end{assumption}
\cref{ass:Theta-non-drecreasing-and-concave} allows us to rely on Jensen's inequality when controlling one of the two terms coming from the separation of force fields (see \cref{sec:Insufficiency of Loeper's approach in the Yudovich setting}). This encompasses for instance the bounded case with $\Theta(r) = 1$  (and $c_{p,\Theta;d} = e^{-\max\{p,d+1\}}$), the exponential Orlicz space with $\Theta(r) = r^{1/\alpha}$ and $1 \le \alpha < +\infty$ (and $c_{p,\Theta;d} = e^{-\max\{p\beta,d+1\}}$, $\beta := 1 + 1/\alpha$), and also a family of iterated logarithms due to Yudovich \cite[Section 3]{yudovich_uniqueness_1995} for two-dimensional Euler's equations in vorticity form; $\Theta_n : [0, +\infty) \to [0, +\infty)$ for $n \ge 1$ ($c_{p,\Theta;d} = \min\{\exp_{n+1}^{-2p}(1),e^{-d-1}\}$) given by 
\begin{equation*}
    \Theta_n(r) := \begin{cases}
        r\abs{\log_1(r)}^2 \abs{\log_2(r)}^2 \cdots \abs{\log_n(r)}^2 &\text{if }r \ge \exp_n(1), \\
        \Theta_n(\exp_n(1)) &\text{else},
    \end{cases}
\end{equation*}
where $\exp_0(1) := 1$, $\exp_{n+1}(1) := e^{\exp_n(1)}$, and
\begin{equation*}
    \log_n(r) := \underbrace{\log \circ \log \circ \cdots \circ \log}_{(n-1) \text{times}} \circ |\log r|.
\end{equation*}
Moreover, each of these cases satisfies the following assumption:
\begin{assumption}\label{ass:Osgood}
    The growth function $\Theta$ is such that the following Osgood condition holds:
    \begin{equation*}
        \int_0^{c_{p,\Theta;d}} \frac{1}{\sqrt{s\varphi_\Theta(s)}} \: ds = +\infty.
    \end{equation*}
    In that case, set
    \begin{equation*}
        \Psi_{p,\Theta;d}(r) := \int_r^{c_{p,\Theta;d}} \frac{1}{\sqrt{s\varphi_\Theta(s)}} \: ds
    \end{equation*}
    and denote by $\Psi^{-1}_{p,\Theta;d}$ its inverse function.
\end{assumption}
\cref{ass:Osgood} is the first-order version of \cite[Assumption 1.15]{crippa_existence_2024} which the authors see as a second-order Osgood condition; it is verified for the examples given in view of the proof of \cite[ Proposition 1.8]{crippa_existence_2024}. In particular, the growth function $\Theta$ cannot increase faster than linear with potential logarithmic corrections; it is asymptotically dominated by any polynomial of degree strictly bigger than $1$, say, $\Theta(r) = o(r^{1 + \delta})$, as can be seen from
\begin{equation*}
    \int_{0^+} \frac{1}{\sqrt{s\varphi_\Theta(s)}} \: ds = -\int_{\infty^-} \frac{1}{\sqrt{r\Theta(r)}} dr \lesssim -\int_{\infty^-} \frac{1}{r^{1+\delta}} dr < +\infty.
\end{equation*}
A useful consequence of this domination is the existence of a constant $C_{p,\Theta} > 0$ that depends only on $p$, $\Theta$, and $d$ such that for all $0 < s < c_{p,\Theta;d}$, it holds
\begin{equation}\label{ineq:Theta kinetic useful}
    \abs{\log\left(\frac{s}{\sqrt{\abs{\log s}\Theta\left(\abs{\log s}\right)}^p}\right)} \le C_{p,\Theta}\abs{\log s}.
\end{equation}
This inequality will arise naturally from the approach based on the kinetic Wasserstein distance; $\sqrt{\abs{\log D_p}\Theta(\abs{\log D_p})}^p$ will be the chosen $\lambda$ creating the nonlinear anisotropy between position and momentum in the \cref{def:kinetic Wasserstein distance} of kinetic Wasserstein distance. In the respective order of the aforementioned cases (bounded, exponential Orlicz, and Yudovich macroscopic density with growth function $\Theta_n$), direct computations give the constants $C_{p,\Theta} = 1 + p/2, \,1 + p\beta/2$, and $1 + p(n+1)$ respectively. Finally, we assume that a variant of \eqref{ineq:Theta kinetic useful} holds, which also appears naturally from this approach.
\begin{assumption}\label{ass:Theta-two-inequalities-for-lambda}
    The growth function $\Theta$ is such that there is a constant $\overline{C}_{p,\Theta} > 0$ that depends only on $p$, $\Theta$, and $d$ such that for all $0 < s < c_{p,\Theta;d}$, it holds
    \begin{equation*}
        \Theta\left(\abs{\log\left(\frac{s}{\sqrt{\abs{\log s}\Theta\left(\abs{\log s}\right)}^p}\right)}\right) \le \overline{C}_{p,\Theta}\Theta\left(\abs{\log s}\right).
    \end{equation*}
\end{assumption}
This assumption is verified for the three usual examples cited above (bounded, exponential Orlicz, and Yudovich macroscopic density with growth function $\Theta_n$) with $\overline{C}_{p,\Theta} = 1, \, (1 + p\beta/2)^{1/\alpha}$, and $(1 + p(n+1)) \times \Pi_{i=0}^n 2^i \log_i(1 + p(n+1))$.

\subsection{Stability estimates} Our main result provides $W_p$ stability estimates in the range $1 \le p < +\infty$ to \eqref{sys:Vlasov--Poisson} with Yudovich macroscopic densities and extends Crippa, Inversi, Saffirio, and Stefani's $W_1$ stability estimate \cite[Theorem 1.6]{crippa_existence_2024} to higher order Wasserstein distances in a unified way.

Because the modulus of continuity for the force field is more singular for Yudovich macroscopic densities compared to the $\log$-Lipschitz modulus of continuity in the bounded case, Loeper's approach is not applicable (see \cref{sec:Insufficiency of Loeper's approach in the Yudovich setting} for details). Instead, we use an appropriate kinetic Wasserstein distance that accommodates the modulus of continuity, making it suitable for Yudovich macroscopic densities.
\begin{theorem}\label{thm:Stability estimate kinetic}
    Let $1 \le p < +\infty$. Let $f_1, f_2$ be two Lagrangian weak solutions to \eqref{sys:Vlasov--Poisson} and $\Theta$ a growth function such that \cref{ass:L^1-bound,ass:varphi_theta-continuous,ass:Theta-non-drecreasing-and-concave,ass:Theta-two-inequalities-for-lambda,ass:Osgood} hold. Then there is a constant $c_{\Theta} > 0$ and a nonnegative function $J(t) \in L^1([0,T])$ both depending only on $p, \Theta$ and on $\X$ such that if $W_p^p(f_1(0), f_2(0))$ is sufficiently small so that $W_p^p(f_1(0), f_2(0)) \le c_{\Theta}$ and
    \begin{equation}\label{ineq:initial condition W_p^p}
        \Psi_{p,\Theta;d}\left(\Phi^{-1}_{p,\Theta}\left[W_p^p\left(f_1(0), f_2(0)\right)\right]\right) \ge \int_0^t J(s) \: ds,
    \end{equation}
    where $\Phi^{-1}_{p,\Theta}$ is the inverse of the map $s \mapsto s/\sqrt{\abs{\log s}\Theta(\abs{\log s})}^p$, then the following stability estimate holds:
    \begin{equation}\label{ineq:stability estimate W_p VP}
        W_p^p\left(f_1(t), f_2(t)\right) \le \Psi_{p,\Theta;d}^{-1}\left(\Psi_{p,\Theta;d}\left(\Phi^{-1}_{p,\Theta}\left[W_p^p\left(f_1(0), f_2(0)\right)\right]\right) - \int_0^t J(s) \: ds\right).
    \end{equation}
\end{theorem}
Let us specify \eqref{ineq:stability estimate W_p VP} explicitly for the aforementioned cases in \cref{sec:Yudovich spaces and assumptions} where $\Phi^{-1}_{p,\Theta}$ behaves like the map $\varphi_{p/2, \Theta} : s \mapsto s\sqrt{\abs{\log s}\Theta(\abs{\log s})}^p$ near the origin: 
\begin{itemize}
    \item The exponential Orlicz space, $\Theta(r) = r^{1/\alpha}$, which yields for $1 < \alpha < +\infty$ ($2 > \beta=1 + 1/\alpha > 1$),
        \begin{multline}\label{ineq:stability estimate Orlicz alpha}
            W_p^p\left(f_1(t), f_2(t)\right) \\ \le \exp\left\{-\left(\sqrt{\abs{\log\Bigg\{W_p^p(f_1(0), f_2(0))\sqrt{\abs{\log W_p^p(f_1(0), f_2(0))}}^{p\beta}\Bigg\}}^{2-\beta}} \!- \left(\frac{2-\beta}{2}\right)\int_0^t J(s) \: ds\right)^{\frac{2}{2-\beta}}\right\}
        \end{multline}
        and for $\alpha = 1$ ($\beta = 2$),
        \begin{equation}\label{ineq:stability estimate Orlicz 1}
            W_p^p\left(f_1(t), f_2(t)\right) \le \exp\Bigg\{\log\left(W_p^p\left(f_1(0), f_2(0)\right)\abs{\log W_p^p\left(f_1(0), f_2(0)\right)}^p\right)\exp\left(-\int_0^t J(s)\: ds\right)\Bigg\}.
        \end{equation}
    \item The bounded space, $\Theta(r) = 1$, corresponding to the limiting case $\alpha \to +\infty$ ($\beta \to 1$) in \eqref{ineq:stability estimate Orlicz alpha}.
    \item The Yudovich space with the family $\Theta_n$,
        \begin{multline}\label{ineq:stability estimate Yudovich family}
            W_p^p\left(f_1(t), f_2(t)\right) \\ \le \left(W_p^p\left(f_1(0), f_2(0)\right)\sqrt{\abs{\log W_p^p\left(f_1(0), f_2(0)\right)}\Theta_n\left(\abs{\log W_p^p\left(f_1(0), f_2(0)\right)}\right)}^p\right) ^ {\exp_{n+1}\left(-\int_0^t J(s) \: ds\right)}.
        \end{multline}
\end{itemize}
\cref{thm:Stability estimate kinetic} positively answers Iacobelli's suggestion \cite[Section 4]{iacobelli_new_2022}: Obtaining a $W_2$ stability estimate for \eqref{sys:Vlasov--Poisson} in the exponential Orlicz setting should be possible based on the kinetic Wasserstein distance approach solely relying on the fact that $(\dot{X}, \dot{V}) = b(X, V)$ for some sufficiently regular function $b$, rather than on the second-order structure $\ddot{X} = -\nabla U_f$ as employed by Holding and Miot.

\section{An \texorpdfstring{$L^p$}{Lp}-estimate via interpolation for Yudovich macroscopic densities}\label{Sec:An Lp-estimate via interpolation for Yudovich macroscopic densities}
As discussed in \cref{sec:The Vlasov--Poisson system}, a crucial aspect of Loeper's optimal transport approach, and, by extension, the kinetic Wasserstein distance approach, is the relationship between the $L^2$-norm of the difference of the force fields and the $2$-Wasserstein distance between their respective macroscopic densities. In the bounded setting, Loeper obtained the following estimate \cite[Theorem 2.9]{loeper_uniqueness_2006}:
\begin{equation}\label{ineq:Loeper L^2 Q}
    \norm{\nabla U_{f_1} - \nabla U_{f_2}}_{L^2(\X)} \le \sqrt{\max\left\{\norm{\rho_{f_1}}_{L^\infty(\X)}, \norm{\rho_{f_2}}_{L^\infty(\X)} \right\}} W_2\left(\rho_{f_1}, \rho_{f_2}\right)  \lesssim \sqrt{Q(t)},
\end{equation}
where $Q(t)$ is the quantity given by
\begin{equation*}
    Q(t) := \int_{\RdRd} \abs{X_1(t;x,v) - X_2(t;y,w)}^2 + \abs{V_1 - V_2}^2 \: d\pi_0(x,v,y,w) 
\end{equation*}
and $\pi_0 \in \Pi(f_1(0), f_2(0))$ is an optimal $W_2$  coupling. 

Here, we consider the following quantity inspired by the kinetic Wasserstein distance:
\begin{align*}
    D_p(t) &= \int_{(\X\times\Rd)^2} \lambda(t)\abs{X_1(t; x, v) - X_2(t; y, w)}^p + \abs{V_1(t; x, v) - V_2(t; y, w)}^p \: d\pi_0(x, v, y, w) \\
        &= \int_{(\X\times\Rd)^2} \lambda(t)\abs{x - y}^p + \abs{v - w}^p \: d\pi_t(x, v, y, w),
\end{align*}
where $\pi_t \in \Pi(f_1(t), f_2(t))$ is such that $\pi_0$ is an optimal $W_p$ coupling (see \cite[Section 4.1]{griffin-pickering_global_2021} for a construction) and the $\lambda(t) := \sqrt{\abs{\log D_p(t)}\Theta(\abs{\log D_p(t)})}^p$ depends itself on $D_p$, whose specific choice, due to optimization considerations, will become apparent in \cref{sec:Proof of the main result}. While $\lambda$ depends on $D_p$, the implicit equation defining $D_p$ remains well-posed, as an adaptation of \cite[Lemma 3.7]{iacobelli_new_2022} ensures this under the condition that $D_p$ is small enough so that $D_p \mapsto \sqrt{\abs{\log D_p}\Theta(\abs{\log D_p})}^p$ is nonincreasing, seen as a function of the quantity $D_p$. This quantity plays a similar role as does Loeper's quantity $Q$ in \cite{loeper_uniqueness_2006}; that is, to control the Wasserstein distance, and it holds
\begin{equation}\label{ineq:control W_p by D_p}
    W_p^p\left(\rho_{f_1}, \rho_{f_2}\right) \le \frac{D_p}{\lambda}, \quad W_p^p\left(f_1, f_2\right) \le D_p.
\end{equation}

We obtain an $L^p$-version of Loeper's $L^2$-estimate \eqref{ineq:Loeper L^2 Q} for Yudovich macroscopic densities in terms of $D_p$:
\begin{proposition}\label{prop:L^p estimate}
    Let $1 < p < +\infty$. Suppose that on $[0, T]$ it holds
    \begin{equation}\label{ineq:first regime}
        \abs{\log\left(\frac{D_p(t)}{\lambda(t)}\right)} > 1.
    \end{equation}
    Then there is a nonnegative function $C_U(t) \in L^\infty([0, T])$ explicitly computable from the proof that only depends on $p$, $\Theta$, $\X$, and on the quantity $D_p(t)$ such that
    \begin{equation}\label{ineq:L^p estimate}
         \norm{\nabla U_{f_1}(X_2) - \nabla U_{f_2}(X_2)}_{L^p(\X)} \\ \le C_U \widetilde A \left(\frac{D_p}{\lambda}\right)^{\frac{1}{p}}\Theta^{\frac{1}{p'}}\left(\abs{\log \left(\frac{D_p}{\lambda}\right)}\right),
    \end{equation}
    where we set
    \begin{equation*}\label{eq:widetilde A}
        \widetilde A := \left(1 + \max\left\{\norm{\rho_{f_1}}_{Y^\Theta(\X)},\norm{\rho_{f_2}}_{Y^\Theta(\X)} \right\}\right)\max\left\{\norm{\rho_{f_1}}_{Y^\Theta(\X)},\norm{\rho_{f_2}}_{Y^\Theta(\X)} \right\}^\frac{1}{p'}.
    \end{equation*}
    In addition, the following estimate holds:
    \begin{equation}\label{ineq:L^p estimate dpi}
        \left(\int_{(\XRd)^2} \abs{\nabla U_{f_1}(X_2) - \nabla U_{f_2}(X_2)}^p \: d\pi_0\right)^{\frac{1}{p}} \le e^{\frac{1}{p}}C_U \overline{A} \left(\frac{D_p}{\lambda}\right)^{\frac{1}{p}}\Theta\left(\abs{\log \left(\frac{D_p}{\lambda}\right)}\right), \quad \overline{A} := \widetilde A \times \norm{\rho_{f_2}}_{Y^\Theta(\X)}^{\frac{1}{p}}.
    \end{equation}
\end{proposition}
\begin{remark}
    The $L^p$-estimate is only valid in a regime on $D_p$; that is, $D_p$ has to be small enough, which is anyway required for it to be well-defined. Here, the regime appears through the definition of the Yudovich norm and is reminiscent of a regime that can be found already in the proof of the $W_p$ stability estimate \cite[proof of Theorem 1.11]{iacobelli_stability_2024} in the bounded setting.  
\end{remark}
In the bounded setting, Iacobelli and the first author \cite[Proposition 1.8]{iacobelli_stability_2024} generalized Loeper's estimate from $L^2$ to $L^p$ for $1 < p < +\infty$ via the Helmholtz-Weyl decomposition \cite[Theorem III.1.2]{galdi_introduction_2011} for $\X=\Td$ or Calderon-Zygmund's inequality \cite[Theorem II.11.4]{galdi_introduction_2011} for $\X=\Rd$ and homogeneous Sobolev spaces on $\X$ (see \cite[Chapter II.6]{galdi_introduction_2011}):

\begin{definition}\label{def:Homogeneous Sobolev}
    The \emph{homogeneous Sobolev space} is the Banach space
    \begin{equation*}
        \dot{W}^{1,p}(\X) := \left\{[g];\quad g \in W^{1,p}(\X), \quad \nabla g \in L^p(\X) \right\},
    \end{equation*}
    where $[\cdot] := \{\cdot + c; \: c\in\R\}$ denotes the equivalence class of functions up to a constant, together with the norm
    \begin{equation*}
        \norm{[g]}_{\dot{W}^{1,p}(\X)} := \norm{\nabla g}_{L^p(\X)}.
    \end{equation*}
    The \emph{dual homogeneous Sobolev space} $\dot{W}^{-1,p}(\X)$ is defined to be the topological dual of $\dot{W}^{1,p'}(\Td)$ equipped with the \emph{strong dual homogeneous Sobolev norm}.
\end{definition}
Given a function $h$ with $\int h = 0$, by density of quotient test functions belonging to $\dot{\D}(\X) := \{[\phi]; \: \phi \in C_c^\infty(\X)\}$ for $1 < p < +\infty$, it holds
\begin{align*}
    \norm{h}_{\dot{W}^{-1, p}(\X)} &:= \sup\left\{\int_{\Td} h\,[g] \: dx; \quad g \in \dot{W}^{1,p'}(\X), \quad \norm{[g]}_{\dot{W}^{1,p'}(\X)} \le 1 \right\} \\
        &= \sup\left\{\int_{\X} h\,[\phi] \: dx; \quad [\phi] \in \dot{\D}(\X),\quad \norm{[\phi]}_{\dot{W}^{1,p'}(\X) \le 1} \right\}.
\end{align*}

The following key lemma relates the Lebesgue $L^p$-norm to the dual homogeneous Sobolev $\dot{W}^{-1,p}$-norm for the difference of densities:
\begin{lemma}[{\cite[Lemma 1.7]{iacobelli_stability_2024}}]\label{lem:Delta E}
    Let $1 < p < +\infty$. Let $\rho_1, \rho_2$ be two probability measures on $\X$ that are absolutely continuous with respect to the Lebesgue measure, and let $U_i$ satisfy $\sigma\Delta U_i = \rho_i - 1_{\X = \Td}$ in the distributional sense with $i = 1,2$. Then there is a constant $C_{\HW} > 0$ that only depends on $p$ and $\X$ such that
    \begin{equation}\label{ineq:L^p dot W^-1,p estimate}
         \norm{\nabla U_1 - \nabla U_2}_{L^p(\X)} \le C_{\HW}\norm{\rho_1 - \rho_2}_{\dot{W}^{-1,p}(\X)}.
    \end{equation}
\end{lemma}
With the help of this \cref{lem:Delta E} and the existence of an optimal transport map adapted to our context, we will be able to prove \cref{prop:L^p estimate}, namely, the $L^p$-estimate.
\begin{theorem}[{Gangbo-McCann \cite[Theorem 1.2]{gangbo_geometry_1996}}]\label{thm:Gangbo-McCann}
    Let $\rho_1, \rho_2$ be two probability measures on $\X$ that are absolutely continuous with respect to the Lebesgue measure. Then
    \begin{equation*}
        W_p\left(\rho_1, \rho_2\right) = \left(\inf_{\T_\# \rho_1 = \rho_2} \bigg\{\int_{\X} \abs{x - \T(x)}^p \: d\rho_1(x) \bigg\}\right)^{\frac{1}{p}},
    \end{equation*}
    where the infimum runs over all measurable mappings $\T : \X \to \X$ that push forward $\rho_1$ onto $\rho_2$. Moreover, the infimum is reached by a $\rho_1(dx)$-almost surely unique mapping $\T$, and there is a $\abs{\cdot}^p$-convex function $\psi$ such that $\T = \Id_{\X} -  \left(\nabla\abs{\cdot}^p\right)^{-1} \circ \nabla\psi$, where we denote $\left(\nabla h^*\right)$ by $\left(\nabla h\right)^{-1}$ for a function $h$ with $h^*$ its Legendre transform.
\end{theorem}
The last two ingredients are the uniform bound on the force fields:
\begin{lemma}[{\cite[Proposition 3.1]{crippa_existence_2024}}]\label{lem:E L^infty}
Let $\rho \in Y^\Theta(\X)$ be a probability measure. Let $U$ satisfy $\sigma\Delta U = \rho - 1_{\X = \Td}$ in the distributional sense. Then there exists a dimensional constant $C_d > 0$ depending on $\X$ such that
\begin{equation}\label{ineq:E L^infty rho L^q estimate}
    \norm{\nabla U}_{L^\infty(\X)} \le C_d\left(1+\norm{\rho}_{Y^\Theta(\X)}\right).
\end{equation}
\end{lemma}
And a standard interpolation lemma:
\begin{lemma}\label{lem:interpolation lemma}
    Let $h \in L^b\cap L^c(\X)$, with $1 \le b < c \le +\infty$. Then for any $b \le a < c$, it holds
    \begin{equation*}
        \norm{h}_{L^a(\X)} \le \norm{h}_{L^b(\X)}^{\frac{1/b-1/c}{1/a-1/c}} \norm{h}_{L^c(\X)}^{\frac{1/a-1/b}{1/a-1/c}}.
    \end{equation*}
\end{lemma}
\begin{proof}
    Write $\abs{h}^a = \abs{h}^d\abs{h}^{d-a}$ with
    \begin{equation*}
        d = ab\frac{1/b - 1/c}{1/a - 1/c}.
    \end{equation*}
    Apply H\"older's inequality to the product with the exponent $b/d$ for the first term.
\end{proof}

We are now set to prove the $L^p$-estimate.
\begin{proof}[Proof of \cref{prop:L^p estimate}]
    Let us denote by 
    \begin{equation*}
        \rho_\theta := \left[(\theta-1)\T+(2-\theta)\mbox{Id}_{\X}\right]_{\#}\rho_1
    \end{equation*}
    the interpolant measure between $\rho_1 := \rho_{f_1}$ and $\rho_2 := \rho_{f_2}$, where $\T$ is the optimal transport map of \cref{thm:Gangbo-McCann}. Let $\phi \in C_c^\infty(\X)$ be a compactly supported test function. By the properties of pushforwards of measures, it follows immediately that
    \begin{equation*}
        \int_{\X}\phi(x) \: d\rho_\theta(x) =\int_{\X}\phi\left((\theta-1)\T(x)+(2-\theta)x\right)d\rho_1(x).
    \end{equation*}
    Note that $\int \rho_1 - \rho_2 = 0$ as both $\rho_1$ and $\rho_2$ are probability measures, and Lebesgue's dominated convergence theorem yields
    \begin{equation*}
        \int_{\X} [\phi](x)\left(\rho_1(x) - \rho_2(x)\right) \: dx = \frac{d}{d\theta}\int_{\X}\phi(x) \:d\rho_\theta(x)=\int_{\X}\nabla \phi\left((\theta-1)\T(x)+(2-\theta)x\right)\cdot(\T(x)-x) \:d\rho_1(x).
    \end{equation*}
    Now, by using H\"older's inequality with respect to the measure $\rho_1$, we get
    \begin{align*}
        \frac{d}{d\theta}\int_{\X}\phi(x) \:d\rho_\theta(x) &\le \left(\int_{\X}\abs{\nabla \phi\left((\theta-1)\T(x)+(2-\theta)x\right)}^{p'}d\rho_1(x)\right)^\frac{1}{p'}\left( \int_{\X}\abs{\T(x)-x}^p \: d\rho_1(x)\right)^{\frac{1}{p}}\\
        &= \left(\int_{\X}\abs{\nabla \phi\left(x\right)}^{p'}d\rho_\theta(x)\right)^{\frac{1}{p'}}\left(\int_{\X}\abs{\T(x)-x}^p d\rho_1(x)\right)^{\frac{1}{p}}. \numberthis\label{eq:intermediate bound}
    \end{align*}
    The second term in the product is exactly $W_p(\rho_1,\rho_2)$ by \cref{thm:Gangbo-McCann}. 
    
    This is where the proof changes from the work of Iacobelli and the first author \cite[Proposition 1.8]{iacobelli_stability_2024}. For the first term in the product, the $L^r$-norm of the interpolant is controlled for all $r \ge 1$ (see \cite[Remark 8]{santambrogio_absolute_2009});
    \begin{equation}
        \norm{\rho_\theta}_{L^r(\X)} \le \max\left\{\norm{\rho_1}_{L^r(\X)},\norm{\rho_2}_{L^r(\X)} \right\}.
    \end{equation}
    Thus, using H\"older's inequality  with respect to the Lebesgue measure and using the fact that $\rho_1,\rho_2 \in Y^\Theta(\X)$, we can write for all $1 < r < +\infty$,
    \begin{align*}
        \left(\int_{\X}\abs{\nabla \phi\left(x\right)}^{p'}d\rho_\theta(x)\right)^\frac{1}{p'} &\leq \norm{\rho_\theta}_{L^r(\X)}^\frac{1}{p'} \left(\int_{\X}\abs{\nabla \phi\left(x\right)}^{p'r'}dx\right)^\frac{1}{p'r'}\\
        & \le \max\left\{\norm{\rho_1}_{L^r(\X)},\norm{\rho_2}_{L^r(\X)} \right\}^\frac{1}{p'}\left(\int_{\X}\abs{\nabla \phi\left(x\right)}^{p'r'}dx\right)^\frac{1}{p'r'}\\
        & \le \Theta^{\frac{1}{p'}}(r)\max\left\{\norm{\rho_1}_{Y^\Theta(\X)},\norm{\rho_2}_{Y^\Theta(\X)} \right\}^\frac{1}{p'}\left(\int_{\X}\abs{\nabla \phi\left(x\right)}^{p'r'}dx\right)^\frac{1}{p'r'}.
    \end{align*}
    We combine the above estimate with \eqref{eq:intermediate bound} which yields
    \begin{equation*}
         \int_{\X} [\phi](x)\left(\rho_1(x) - \rho_2(x)\right) \: dx \le \Theta^{\frac{1}{p'}}(r)\max\left\{\norm{\rho_1}_{Y^\Theta(\X)},\norm{\rho_2}_{Y^\Theta(\X)} \right\}^\frac{1}{p'}\left(\int_{\X}\abs{\nabla \phi\left(x\right)}^{p'r'}dx\right)^\frac{1}{p'r'}W_p\left(\rho_1, \rho_2\right),
    \end{equation*}
    and by restricting to quotient test functions $\left[\phi\right] \in \dot\D(\X)$ such that $\norm{\nabla \left[\phi\right]}_{L^{p'r'}(\X)} \leq 1$, recalling the definition of the dual homogeneous Sobolev norm, we obtain
    \begin{equation*}\label{ineq:dot W^1,p'r' estimate}
        \norm{\rho_1-\rho_2}_{\dot{W}^{-1,\frac{pr}{r+p-1}}(\X)} \le \Theta^{\frac{1}{p'}}(r)\max\left\{\norm{\rho_1}_{Y^\Theta(\X)},\norm{\rho_2}_{Y^\Theta(\X)} \right\}^\frac{1}{p'}W_p(\rho_1,\rho_2),
    \end{equation*}
    because $(p'r')' := pr/(r+p-1) > 1$ as Lebesgue's conjugate exponent. Applying the estimate \eqref{ineq:L^p dot W^-1,p estimate} of \cref{lem:Delta E}, we obtain for all $1 < r < +\infty$,
    \begin{equation}\label{ineq:L^(r'p')' estimate}
        \norm{\nabla U_1- \nabla U_2}_{L^\frac{pr}{r+p-1}(\X)} \le C_{\HW;\,(p'r')'}\, \Theta^{\frac{1}{p'}}(r)\max\left\{\norm{\rho_1}_{Y^\Theta(\X)},\norm{\rho_2}_{Y^\Theta(\X)} \right\}^\frac{1}{p'}W_p(\rho_1,\rho_2),
    \end{equation}
    where the subscript in prefactor $C_{\HW}$ specifies that it depends on the Lebesgue exponent $(p'r')'$. Recall that the $L^\infty$-norms of both (negative of) force fields $\nabla U_1 := \nabla U_{f_1}$ and $\nabla U_2 := \nabla U_{f_2}$ are controlled by \eqref{ineq:E L^infty rho L^q estimate} thanks to \cref{lem:E L^infty}. Now, since $(p'r')' := pr/(r+p-1) < p < +\infty$, we can use \cref{lem:interpolation lemma} to interpolate \eqref{ineq:L^(r'p')' estimate} and \eqref{ineq:E L^infty rho L^q estimate} for
    \begin{equation*}
        \frac{1/p - 1/\infty}{1/(p'r')' - 1/\infty} = \frac{r}{r + p -1}, \quad \frac{1/(p'r')' - 1/p}{1/(p'r')' - 1/\infty} = \frac{p-1}{r + p -1}
    \end{equation*} to deduce
    \begin{align*}
        \norm{\nabla U_{f_1} - \nabla U_{f_2}}_{L^p(\X)}^p &\le \norm{\nabla U_{f_1} - \nabla U_{f_2}}_{L^\frac{pr}{r+p-1}(\X)}^\frac{pr}{r+p-1}\norm{\nabla U_{f_1} - \nabla U_{f_2}}_{L^\infty(\X)}^\frac{p(p-1)}{r+p-1} \\
            &\le \widetilde C_U^p\left(1 + \max\left\{\norm{\rho_{f_1}}_{Y^\Theta(\X)}, \norm{\rho_{f_2}}_{Y^\Theta(\X)}\right\}\right) \\
            &\quad \times \left(\Theta^{\frac{p}{p'}}(r)\max\left\{\norm{\rho_{f_1}}_{Y^\Theta(\X)},\norm{\rho_{f_2}}_{Y^\Theta(\X)} \right\}^\frac{p}{p'}W^p_p\left(\rho_{f_1},\rho_{f_2}\right)\right)^\frac{r}{r+p-1}\\
            &\le \widetilde C_U^p\widetilde A^p \Theta^{\frac{p}{p'}}(r)\left[W_p^p\left(\rho_{f_1},\rho_{f_2}\right)\right]^\frac{r}{r+p-1}, \numberthis\label{ineq:E part 1 intermediate step 1}
    \end{align*}
    where we used that both $\Theta(r) \ge 1$ and $\norm{\rho_{f_i}}_{Y^\Theta(\X)} \ge 1$ for $i = 1,2$ since it was assumed that $\Theta(1) = 1$ for convenience, and set 
    \begin{equation*}
        \widetilde C_U := 2C_d C_{\HW;\,(p'r')'}, \quad \widetilde A := \left(1 + \max\left\{\norm{\rho_{f_1}}_{Y^\Theta(\X)}, \norm{\rho_{f_2}}_{Y^\Theta(\X)}\right\}\right)\max\left\{\norm{\rho_{f_1}}_{Y^\Theta(\X)}, \norm{\rho_{f_2}}_{Y^\Theta(\X)}\right\}^{\frac{1}{p'}}.
    \end{equation*}
    Recall that $D_p$ controls the $p$-Wasserstein distance as claimed in \eqref{ineq:control W_p by D_p}; indeed, since $(X_1, X_2)_{\#}\pi_0$ has marginals $\rho_{f_1}$ and $\rho_{f_2}$,
    \begin{multline*}
        W_p^p\left(\rho_{f_1}, \rho_{f_2}\right) := \inf_{\gamma \in \Pi(\rho_{f_1}, \rho_{f_2})} \int_{\X\times\X} \abs{x-y}^p \: d\gamma(x,y) 
            \le \int_{\X\times \X} \abs{x-y}^p \: d\Big[\left(X_1, X_2\right)_\#\pi_0\Big](x,y) \\
            = \int_{(\XRd)^2} \abs{X_1 - X_2}^p \: d\pi_0(x,v,y,w) \le \frac{D_p}{\lambda}.
    \end{multline*}
    Hence, \eqref{ineq:E part 1 intermediate step 1} becomes
    \begin{equation}\label{ineq:E part 1 intermediate step 2}
        \norm{\nabla U_{f_1} - \nabla U_{f_2}}_{L^p(\X)}^p \le \widetilde C_U^p \widetilde A^p \Theta^{\frac{p}{p'}}(r)\left(\frac{D_p}{\lambda}\right)^\frac{r}{r+p-1}.
    \end{equation}
    We choose
    \begin{equation*}
        r := \abs{\log \left(\frac{D_p}{\lambda}\right)},
    \end{equation*}
    which is required to be strictly between $1$ and $+\infty$. If $D_p/\lambda = 0$, then, since it controls the $p$-Wasserstein distance between the macroscopic densities, we have $\nabla U_{f_1} = \nabla U_{f_2}$ by construction, and the proof is complete. Therefore, without loss of generality, we may assume throughout the proof that $D_p(t)/\lambda(t) > 0$ on $[0, T]$, and in particular $r(t) < +\infty$ on $[0, T]$. The regime dictated on $r(t)$ thus translates to
    \begin{equation*}
        \abs{\log \left(\frac{D_p(t)}{\lambda(t)}\right)} > 1
    \end{equation*}
    and coincides with the regime \eqref{ineq:first regime} of the statement of \cref{prop:L^p estimate}. In this regime, the particular choice of $r$ gives 
    \begin{align*}
        \left(\frac{D_p}{\lambda}\right)^{\frac{r}{r+p-1}} &= \left(\frac{D_p}{\lambda}\right)\left(\frac{D_p}{\lambda}\right)^{\frac{-(p-1)}{r+p-1}} \\
            &= \left(\frac{D_p}{\lambda}\right)\exp\left(\frac{p-1}{\abs{\log\left(\frac{D_p}{\lambda}\right)} + p - 1}\abs{\log\left(\frac{D_p}{\lambda}\right)}\right) \le e^{p-1}\left(\frac{D_p}{\lambda}\right),
    \end{align*}
    and combining with \eqref{ineq:E part 1 intermediate step 2}, this yields the $L^p$-estimate \eqref{ineq:L^p estimate};
    \begin{equation*}
        \norm{\nabla U_{f_1} - \nabla U_{f_2}}_{L^p(\X)}^p \\
            \le C^p_U \widetilde{A}^p \left(\frac{D_p}{\lambda}\right)\Theta^{\frac{p}{p'}}\left(\abs{\log\left(\frac{D_p}{\lambda}\right)}\right)
    \end{equation*}
    with $C_U := e^{1/p'}\widetilde C_U$ depending only on $p$, $\Theta$, $\X$, and on the quantity $D_p$ through $C_{\HW;(p'r')'}$. 

    For the second estimate \eqref{ineq:L^p estimate dpi}, recall that $\pi_t$ has marginals $f_1(t) = Z_1(t)_{\#}f_1(0)$ and $f_2(t) = Z_2(t)_{\#}f_2(0)$ so that for all $\phi \in C_b((\XRd)^2)$, it holds
    \begin{equation*}\label{eq:marginal property}
        \int_{\left(\XRd\right)^2} \phi\left(Z_1(t; x, v), Z_2(t; y, w)\right) \: d\pi_0(x, v, y, w) = \int_{\left(\XRd\right)^2} \phi(x, v, y, w) \: d\pi_t(x,v,y,w). 
    \end{equation*}
    In particular,
    \begin{multline*}
        \left(\int_{(\XRd)^2} \abs{\nabla U_{f_1}(t; X_2(t;y,w)) - \nabla U_{f_2}(t; X_2(t;y,w))}^p \: d\pi_0\right)^{\frac{1}{p}} 
            \\ = \left(\int_{\X} \abs{\nabla U_{f_1}(t; y) - \nabla U_{f_2}(t; y)}^p \rho_{f_2}(t; y) \: dy\right)^{\frac{1}{p}}.
    \end{multline*}
    By H\"older's inequality with respect to the Lebesgue measure,
    \begin{equation}\label{ineq:Holder nabla U_f1 nabla U_f2 rho_f2}
        \left(\int_{\X} \abs{\nabla U_{f_1}(y) - \nabla U_{f_2}(y)}^p \rho_{f_2}(y) \: dy\right)^{\frac{1}{p}} \le \norm{\nabla U_{f_1} - \nabla U_{f_2}}_{L^{pr'}(\X)} \Theta^{\frac{1}{p}}(r) \norm{\rho_{f_2}}_{Y^\Theta(\X)}^{\frac{1}{p}}.
    \end{equation}
    Since $1 < r := \abs{\log(D_p/\lambda)} < +\infty$; $p < pr' < +\infty$, we can use \cref{lem:interpolation lemma} to interpolate between $p$ and $+\infty$ for
    \begin{equation*}
        \frac{1/(pr') - 1/+\infty}{1/p - 1/+\infty} = \frac{1}{r'}, \quad \frac{1/p - 1/(pr')}{1/(pr') - 1/+\infty} = \frac{1}{r}
    \end{equation*}
    to get
    \begin{align*}
        \norm{\nabla U_{f_1} - \nabla U_{f_2}}_{L^{pr'}(\X)} &\le \norm{\nabla U_{f_1} - \nabla U_{f_2}}_{L^{p}(\X)}^{\frac{1}{r'}} \norm{\nabla U_{f_1} - \nabla U_{f_2}}_{L^\infty(\X)}^{\frac{1}{r}} \\
            &\le \Bigg[C_U \widetilde A\left(\frac{D_p}{\lambda}\right)^{\frac{1}{p}}\Theta^{\frac{1}{p'}}\left(\abs{\log \left(\frac{D_p}{\lambda}\right)}\right)\Bigg]^{\frac{1}{r'}}\\
            &\quad \times \Big[2C_d \left(1 + \max\left\{\norm{\rho_{f_1}}_{Y^\Theta(\X)},  \norm{\rho_{f_2}}_{Y^\Theta(\X)}\right\}\right)\Big]^{\frac{1}{r}} \\
            &\le C_U \widetilde A\left(\frac{D_p}{\lambda}\right)^{\frac{1}{p}\left(1-\frac{1}{r}\right)}\Theta^{\frac{1}{p'}}\left(\abs{\log \left(\frac{D_p}{\lambda}\right)}\right), \numberthis\label{ineq:explanation wlog 2 part 1}
    \end{align*}
    courtesy of the $L^p$ and $L^\infty$ estimates \eqref{ineq:L^p estimate} and \eqref{ineq:E L^infty rho L^q estimate} respectively, in addition to $\Theta(r) \ge 1$ and $\norm{\rho_{f_i}}_{Y^\Theta(\X)} \ge 1$ for $i = 1,2$. Recalling the definition of $\overline{A} := \widetilde A \times \norm{\rho_{f_2}}_{Y^\Theta(\X)}^{1/p}$ and combining \eqref{ineq:Holder nabla U_f1 nabla U_f2 rho_f2} with \eqref{ineq:explanation wlog 2 part 1} and $(D_p/\lambda)^{1 - 1/r} = e(D_p/\lambda)$ yields the $L^p$-estimate \eqref{ineq:L^p estimate dpi}.

    To conclude, note that since $1 < p <+\infty$ with $1 < r(t) \,(\le +\infty)$ in the regime \eqref{ineq:first regime}, we have that $1 < (p'r'(t))' \le p < +\infty$ on $[0, T]$. For the torus case $\X = \Td$, this implies $C_{\HW;\,(p'r'(t))'} \in L^\infty([0, T])$, and, consequently, $C_U(t) \in L^\infty([0, T])$. Here, the Helmholtz-Weyl prefactor is defined as $C_{\HW';\,(p'r')'} := 1 + C_{P_{(p'r')'}}$ (see proof of \cite[Lemma 1.7]{iacobelli_stability_2024}), where $C_{P_{(p'r')'}}$ is the operator norm of the Helmholtz-Weyl (or Leray) projection operator onto the solenoidal hydrodynamic space. This norm is finite for all Lebesgue exponents except $1$ and $+\infty$ due to the boundedness of the projection \cite[Remark III.1.1]{galdi_introduction_2011}. For the Euclidean case $\X = \Rd$, $C_{\HW;\,(p'r')'}$ is the constant arising from Calderon-Zygmund's inequality (see proof of \cite[Lemma 1.7]{iacobelli_stability_2024}), which is also finite for all Lebesgue exponents different than 1 and $+\infty$ \cite[Theorem II.11.4]{galdi_introduction_2011}.
\end{proof}
\begin{remark}
    We emphasize that even if one is solely interested in the $W_2$ stability estimate for \eqref{sys:Vlasov--Poisson} with Yudovich macroscopic densities, the $L^2$-estimate \cref{prop:L^p estimate} is based on \cref{lem:Delta E} for $p\ne 2$ due to the definition of the Yudovich norm and the use of interpolation. \cref{lem:Delta E} relates the $L^p$-norm to the dual homogeneous Sobolev norm $\dot{W}^{-1,p}$ and requires either the  Helmholtz-Weyl decomposition or the Calderon-Zygmund inequality. However, in the bounded setting, Loeper \cite[Proposition 3.1]{loeper_uniqueness_2005} provides an alternative proof of the $L^2$-estimate which circumvents the use of the dual homogeneous Sobolev norm by relying on the Benamou-Brenier formula \cite[Proposition 1.1]{benamou_computational_2000} and a carefully chosen test function. 
\end{remark}

\section{Stability estimates in \texorpdfstring{$p$}{p}-kinetic Wasserstein distances}
\subsection{Insufficiency of Loeper's approach in the Yudovich setting}\label{sec:Insufficiency of Loeper's approach in the Yudovich setting}
The classical approach established by Loeper falls short when addressing macroscopic densities whose Lebesgue norms increase at most according to a prescribed growth function, which is the Yudovich setting. 

One studies the evolution of the following quantity, inspired by Loeper, which controls the $2$-Wasserstein distance between two solutions:
\begin{equation*}
    Q(t) = \int_{(\XRd)^2} \abs{X_1 - X_2}^2 + \abs{V_1 - V_2}^2 \: d\pi_0.
\end{equation*}
After taking the time derivative and applying Cauchy-Schwartz inequality, one is left with
\begin{equation*}
    \frac{1}{2}\dot{Q} \le Q + \sqrt{Q}\sqrt{\int_{(\XRd)^2} \abs{\nabla U_{f_1}(X_1) - \nabla U_{f_2}(X_2)}^2 \:d\pi_0}.
\end{equation*}
Now, the classical step involves separating the force fields into two parts;
\begin{equation*}
    \abs{\nabla U_{f_1}(t;X_1) - \nabla U_{f_2}(t;X_2)} \le \abs{\nabla U_{f_1}(t;X_1) - \nabla U_{f_1}(t;X_2)} + \abs{\nabla U_{f_1}(t;X_2) - \nabla U_{f_2}(t;X_2)}.
\end{equation*}
The first term exhibits the regularity of a force field evaluated along two possibly distinct characteristics, which depends on the regularity of the associated macroscopic density. The second term compares the two force fields along the same characteristic.

In the bounded setting, the $\log$-Lipschitz regularity \cite[Lemma 3.1]{loeper_uniqueness_2006}, combined with an adapted nondecreasing concave function and Jensen's inequality, leads to the following after some manipulations:
\begin{equation}\label{ineq:log-Lipschitz estimate bounded}
    \sqrt{\int_{(\XRd)^2} \abs{\nabla U_{f_1}(X_1) - \nabla U_{f_1}(X_2)}^2 d\pi_0} \lesssim \sqrt{Q}\abs{\log Q}.
\end{equation}
Recalling Loeper's $L^2$-estimate \eqref{ineq:Loeper L^2 Q} for the second term of the splitting and combining the estimates gives
\begin{equation*}
    \dot{Q} \lesssim Q\abs{\log Q},
\end{equation*}
which is a closable differential inequality.

In the Yudovich setting, the force field has a modulus of continuity $\varphi_\Theta$ depending on the growth function $\Theta$ so that applying Jensen's inequality leads this time to the following estimate:
\begin{equation}\label{ineq:varphi_theta estimate Yudovich}
    \sqrt{\int_{(\XRd)^2} \abs{\nabla U_{f_1}(X_1) - \nabla U_{f_1}(X_2)}^2 d\pi_0} \lesssim \sqrt{Q}\abs{\log Q}\Theta\Big(\abs{\log Q}\Big).
\end{equation}
Compared to \eqref{ineq:log-Lipschitz estimate bounded}, \eqref{ineq:varphi_theta estimate Yudovich} introduces a factor $\Theta(\abs{\log Q})$. If the modulus of continuity is too strong, as in the case of $\Theta(r) := r^{1/\alpha}$, corresponding to the exponential Orlicz space studied by Holding and Miot in \cite{holding_uniqueness_2018}, it may prevent the closability of the differential inequality.
\subsection{Proof of the main result}\label{sec:Proof of the main result}
We first state a refined version of the modulus of continuity of the force fields with Yudovich macroscopic densities due to Crippa, Inversi, Saffirio and Stefani (see Miot \cite[Proposition 2.1]{miot_uniqueness_2016} and Holding and Miot \cite[Lemma 2.1]{holding_uniqueness_2018} in the exponential Orlicz setting, and Loeper \cite[Lemma 3.1]{loeper_uniqueness_2006} in the bounded setting):
\begin{lemma}[{\cite[Lemma 1.1 \& Proposition 3.1]{crippa_existence_2024}}]\label{lem:modulus of continuity}
    Let $\rho \in Y^\Theta(\X)$ be a probability measure. Let $U$ satisfy $\sigma\Delta U = \rho - 1_{\X = \Td}$ in the distributional sense. Then $\nabla U$ has modulus of continuity $\varphi_\Theta$. More precisely, there is a dimensional constant $C_d > 0$ depending on $\X$ such that for all $x,y \in \X$,
    \begin{equation*}
        \int_{\X} \abs{\nabla G(x-q) - \nabla G(y-q)}\rho(q) \: dq \le C_d\left(1 + \norm{\rho}_{Y^\Theta(\X)}\right)\varphi_\Theta(\abs{x - y}),
    \end{equation*}
    where $G$ denotes the Green function on $\X$ associated to the Laplace equation.
\end{lemma}
We also recall Osgood's lemma in this context, which enables us to close the differential inequality involving $D_p$:
\begin{lemma}[{\cite[Lemma 3.4 \& Corollary 3.5]{bahouri_fourier_2011}}]\label{lem:Osgoog}
    Let $\Phi(s) := \sqrt{s\varphi_\Theta(s)}$ satisfies the Osgood condition, namely, \cref{ass:Osgood} holds. Let $G,H$ be two measurable functions such that
    \begin{equation*}
        G(t) \le G(0) + \int_0^t H(s)\Phi(G(s)) \: ds. 
    \end{equation*}
    If $G(0)$ is such that
    \begin{equation*}
        \int_0^T H(s) \: ds \le \Psi_{p,\Theta;d}\left(G(0)\right),
    \end{equation*}
    then on $[0, T]$ it holds 
    \begin{equation*}
        G(t) \le \Psi_{p,\Theta;d}^{-1}\left(\Psi_{p,\Theta;d}\left(G(0)\right) - \int_0^T H(s) \: ds\right).
    \end{equation*}
\end{lemma}
\begin{remark}
    The $\lambda$ that appears in the definition of $D_p$ allows for the redistribution of the excessive contribution of the force fields' modulus of continuity, balancing the roles of the position and momentum variables during the evolution of the quantity $D_p$. Inspecting the proof of \cref{thm:Stability estimate kinetic}, one is left with
    \begin{equation*}
        \dot D_p \lesssim \lambda^{\frac{1}{p}}D_p + \lambda^{-\frac{1}{p}}D_p\abs{\log D_p}\Theta\left(\abs{\log D_p}\right).
    \end{equation*}
    The choice $\lambda := \sqrt{\abs{\log D_p}\Theta(\abs{\log D_p})}^p$ equates both terms in the right-hand side.
\end{remark}
\begin{proof}[Proof of \cref{thm:Stability estimate kinetic}]\label{proof of thm}
    \begin{align*}
    \frac{1}{p}\dot{D}_p = &\:\frac{1}{p} \int_{(\XRd)^2} \dot{\lambda}\abs{X_1 - X_2}^p \: d\pi_0 \\
        &+ \int_{(\XRd)^2} \lambda\abs{X_1 - X_2}^{p-2}(X_1 - X_2)\cdot(V_1 - V_2) \: d\pi_0 \\
        &+ \int_{(\XRd)^2} \abs{V_1 - V_2}^{p-2}(V_1 - V_2)\cdot\left(\nabla U_{f_2} - \nabla U_{f_1}\right) \: d\pi_0.
    \end{align*}
    Using H\"older's inequality with respect to the measure $\pi_0$ for the last two terms, we get
    \begin{equation}\label{ineq:dot D_p before T_1 T_2}
        \dot{D}_p \le \frac{1}{p} \int_{(\XRd)^2} \dot{\lambda}\abs{X_1 - X_2}^p \: d\pi_0  + \lambda^{\frac{1}{p}}D_p + D_p^{\frac{1}{p'}} \left(\int_{(\XRd)^2} \abs{\nabla U_{f_1} - \nabla U_{f_2}}^p \: d\pi_0\right)^{\frac{1}{p}}.
    \end{equation}
    We proceed to the usual splitting of the difference of electric fields;
    \begin{equation}\label{ineq:electric field split}
        \left(\int_{(\XRd)^2} \abs{\nabla U_{f_1} - \nabla U_{f_2}}^p \: d\pi_0\right)^{\frac{1}{p}} \le T_1 + T_2,
    \end{equation}
    where
    \begin{equation*}\label{eq:T_1 and T_2}
        T_1 := \left(\int_{(\XRd)^2} \abs{\nabla U_{f_1}(X_1) - \nabla U_{f_1}(X_2)}^p \:d\pi_0\right)^{\frac{1}{p}}, \quad 
        T_2 := \left(\int_{(\XRd)^2} \abs{\nabla U_{f_1}(X_2) - \nabla U_{f_2}(X_2)}^p \: d\pi_0\right)^{\frac{1}{p}}.
    \end{equation*}
    First, since $\varphi_\Theta$ is the modulus of continuity of $\nabla U_{f_1}$, by \cref{lem:modulus of continuity}, we can bound $T_1$ as follows:
    \begin{align*}
        T_1 &\le C_d\left(1 + \norm{\rho_{f_1}}_{Y^\Theta(\X)}\right) \left(\int_{(\XRd)^2} \varphi^p_\Theta(\abs{X_1 - X_2}) \: d\pi_0\right)^{\frac{1}{p}} \\
            &\le C_d\left(1 + \norm{\rho_{f_1}}_{Y^\Theta(\X)}\right) \left(\left[\int_{\abs{X_1 - X_2}^p < c_{p,\Theta;d}} + \int_{\abs{X_1 - X_2}^p \ge c_{p,\Theta;d}}\right] \varphi_\Theta^p\left(\abs{X_1 - X_2}\right) \: d\pi_0\right)^{\frac{1}{p}}.
    \end{align*}
    For $\abs{X_1 - X_2}^p < c_{p,\Theta;d}$, we observe that
    \begin{align*}
        \varphi^p_\Theta\left(\abs{X_1 - X_2}\right) &= \frac{1}{p^p}\abs{X_1 - X_2}^p \abs{\log \abs{X_1 - X_2}^p}^p \Theta^p\left(\frac{1}{p}\abs{\log \abs{X_1 - X_2}^p}\right) \\
            &\le \frac{1}{p^p} \varphi_{p,\Theta}\left(\abs{X_1-X_2}^p\right),
    \end{align*}
    and for $\abs{X_1-X_2}^p \ge c_{p,\Theta;d}$, we set $C_{p,\varphi_\Theta;d} :=  \max\{\varphi_\Theta(s); \: s \ge c_{p,\Theta;d}^{1/p}\}/c_{p,\Theta;d}^{1/p}$; the bound on $T_1$ becomes
    \begin{align*}
        T_1 \le &C_d\left(1 + \norm{\rho_{f_1}}_{Y^\Theta(\X)}\right)\Bigg(\frac{1}{p} \left[\int_{\abs{X_1 - X_2}^p < c_{p,\Theta;d}} \varphi_{p,\Theta}\left(\abs{X_1-X_2}^p\right) \:d\pi_0\right]^{\frac{1}{p}} \\
            &\quad + C_{p,\varphi_\Theta;d} \left[\int_{\abs{X_1 - X_2}^p \ge c_{p,\Theta;d}} \abs{X_1 - X_2}^p \:d\pi_0\right]^{\frac{1}{p}}\Bigg) \\
            &\le C_d\left(1 + \norm{\rho_{f_1}}_{Y^\Theta(\X)}\right)\left(\frac{1}{p}\left[\varphi_{p,\Theta}\left(\frac{D_p}{\lambda}\right)\right]^{\frac{1}{p}} + C_{p,\varphi_\Theta;d}\left(\frac{D_p}{\lambda}\right)^{\frac{1}{p}}\right)
    \end{align*}
    after applying Jensen's inequality recalling that $\varphi_{p,\Theta}$ is nondecreasing concave on $[0, c_{p,\Theta;d})$ under the \cref{ass:Theta-non-drecreasing-and-concave}.
    Thus, in the regime 
    \begin{equation}\label{ineq:regime T_1}
        \frac{D_p(t)}{\lambda(t)} < c_{p,\Theta;d},
    \end{equation}
    we have
    \begin{equation}\label{ineq:T_1 intermediate estimate}
        T_1 \le C_d\left(1 + \norm{\rho_{f_1}}_{Y^\Theta(\X)}\right) \left(\frac{1}{p} + C_{p,\varphi_\Theta;d}\right) \times \left(\frac{D_p}{\lambda}\right)^{\frac{1}{p}} \abs{\log\left(\frac{D_p}{\lambda}\right)}\Theta\left(\abs{\log\left(\frac{D_p}{\lambda}\right)}\right).
    \end{equation}
    Second, we estimate $T_2$ and distinguish between two cases:
    \newline For $1 < p <+\infty$, the $L^p$-estimate \eqref{ineq:L^p estimate dpi} from \cref{prop:L^p estimate} yields
    \begin{equation}\label{ineq:T_2 intermediate estimate 1<p<infty}
        T_2 \le e^{\frac{1}{p}} C_U \overline{A} \left(\frac{D_p}{\lambda}\right)^{\frac{1}{p}}\Theta\left(\abs{\log \left(\frac{D_p}{\lambda}\right)}\right)
    \end{equation}
    in the regime 
    \begin{equation}\label{ineq:regime T_2}
        \abs{\log\left(\frac{D_p(t)}{\lambda(t)}\right)} > 1.
    \end{equation}
    For $p=1$, we write the difference of the two force fields as a difference of convolutions between the gradient of the Green function $\nabla G$ to the Laplace equation and the macroscopic densities (see also the proof of \cite[Theorem 2.8]{crippa_existence_2024} and of \cite[Lemma 2.2]{holding_uniqueness_2018}). Recall that $\pi_0$ has marginals $f_1(0)$ and $f_2(0)$ and that $f_i(t) = Z_i(t;\cdot, \cdot)_{\#}f_i(0)$ for $i=1,2$. Thus,
    \begin{align*}
        &\abs{\nabla U_{f_1}(t; X_2(t;y,w)) - \nabla U_{f_2}(t; X_2(t;y,w))} \\
        &\:= \abs{ \int_{\XRd} \nabla G\left(X_2(t;y,w) - \tilde{x}\right) \: df_1(t;\tilde{x},\tilde{v}) - \int_{\XRd} \nabla G\left(X_2(t;y,w) - \tilde{y}\right) \: df_2(t;\tilde{y},\tilde{w})} \\
        &\:= \abs{ \int_{\XRd} \nabla G\left(X_2(t;y,w) - X_1(t;\tilde{x},\tilde{v})\right) \: df_1(0;\tilde{x},\tilde{v}) - \int_{\XRd} \nabla G\left(X_2(t;y,w) - X_2(t;\tilde{y},\tilde{w})\right) \: df_2(0;\tilde{y},\tilde{w})} \\
        &\:= \abs{\int_{(\XRd)^2} \nabla G\left(X_2(t;y,w) - X_1(t;\tilde{x},\tilde{v})\right) - \nabla G\left(X_2(t;y,w) - X_2(t;\tilde{y},\tilde{w})\right) \: d\pi_0(\tilde{x}, \tilde{v}, \tilde{y}, \tilde{w})},
    \end{align*}
    and applying Fubini's theorem to $T_2$ gives
    \begin{align*}
        T_2 &\le \int_{(\XRd)^2}\left(\int_{(\XRd)^2} \left|\nabla G\left(X_2(t;y,w) - X_1(t;\tilde{x},\tilde{v})\right) \right.\right. \\
        &\quad - \left. \vphantom{\int_{(\XRd)^2}}\left.\nabla G\left(X_2(t;y,w) - X_2(t;\tilde{y},\tilde{w})\right)\right| \: d\pi_0(x,v,y,w)\right) d\pi_0(\tilde{x}, \tilde{v}, \tilde{y}, \tilde{w}) \\
        &= \int_{(\XRd)^2}\left( \int_{\X} \abs{\nabla G\left(t;y - X_1(t;\tilde{x},\tilde{v}\right) - \nabla G\left(t;y - X_2(t;\tilde{y},\tilde{w}\right)} \rho_{f_2}(t; y) \:dy\right) d\pi_0(\tilde{x}, \tilde{v}, \tilde{y}, \tilde{w}).
    \end{align*}
    By \cref{lem:modulus of continuity},
    \begin{equation*}
        T_2 \le C_d\left(1 + \norm{\rho_{f_2}}_{Y^\Theta(\X)}\right) \int_{(\XRd)^2} \varphi_\Theta\left(\abs{X_1 - X_2}\right) \: d\pi_0,
    \end{equation*}
    which is estimated in the same way as $T_1$ to obtain
    \begin{equation}\label{ineq:T_2 intermediate estimate p=1}
        T_2 \le C_d\left(1 + \norm{\rho_{f_2}}_{Y^\Theta(\X)}\right) \left(\frac{D_1}{\lambda}\right) \abs{\log \left(\frac{D_1}{\lambda}\right)} \Theta\left(\abs{\log \left(\frac{D_1}{\lambda}\right)}\right).
    \end{equation}
    
    Now that we estimated both $T_1$ and $T_2$, we consider the regime
    \begin{equation}\label{ineq:strict regime}
        D_p(t) < c_{p,\Theta;d},
    \end{equation}
    which is more restrictive than \eqref{ineq:regime T_1} and \eqref{ineq:regime T_2}, and for which both $\abs{\log D_p}$ and $\lambda := \sqrt{\abs{\log D_p}\Theta(\abs{\log D_p})}^p$ are greater than $1$. In particular, all of \eqref{ineq:T_1 intermediate estimate}, \eqref{ineq:T_2 intermediate estimate 1<p<infty}, and \eqref{ineq:T_2 intermediate estimate p=1} remain valid estimates within the regime defined by \eqref{ineq:strict regime}. Moreover, under \eqref{ineq:strict regime}, \eqref{ineq:Theta kinetic useful} and \cref{ass:Theta-two-inequalities-for-lambda} yield
    \begin{equation*}
        \abs{\log\left(\frac{D_p}{\lambda}\right)} := \abs{\log\left(\frac{D_p}{\sqrt{\abs{\log D_p}\Theta\left(\abs{\log D_p}\right)}^p}\right)} \le C_{p,\Theta}\abs{\log D_p}, \quad \Theta\left(\abs{\log \left(\frac{D_p}{\lambda}\right)}\right) \le \overline{C}_{p,\Theta}\Theta\left(\abs{\log D_p}\right).
    \end{equation*}
    Applying the above to the estimates \eqref{ineq:T_1 intermediate estimate} of $T_1$ and \eqref{ineq:T_2 intermediate estimate 1<p<infty}, \eqref{ineq:T_2 intermediate estimate p=1} of $T_2$ respectively, and combining them with \eqref{ineq:dot D_p before T_1 T_2} and \eqref{ineq:electric field split}, the differential inequality for $D_p$ becomes 
    \begin{equation*}
        \frac{1}{p}\dot D_p \le \frac{1}{p}\int_{(\XRd)^2} \dot\lambda\abs{X_1 - X_2} \:d\pi_0 + \lambda^{\frac{1}{p}}D_p + C_{p,\Theta;d}A\left(1_{p=1} + 1_{1<p<+\infty}\right)e^{\frac{1}{p}}C_U \times \lambda^{-\frac{1}{p}}D_p\abs{\log D_p}\Theta\left(\abs{\log D_p}\right)
    \end{equation*}
    with $C_{p,\Theta;d} := 2(1/p + C_{p,\varphi_\Theta;d}) \times C_{p,\Theta}\overline{C}_{p,\Theta}$ and $A$ defined in \cref{ass:L^1-bound}. Note that $\lambda$ is non-increasing given that $\dot D_p \ge 0$; indeed, then
    \begin{equation*}
        \dot{\lambda} = \frac{p}{2}\sqrt{\left(\abs{\log D_p}\Theta\left(\abs{\log D_p}\right)\right)^{p-2}} \times \left(\frac{-\dot{D}_p}{D_p}\right)\left(\Theta\left(\abs{\log D_p}\right) + \abs{\log D_p}\dot\Theta\left(\abs{\log D_p}\right)\right) \le 0.
    \end{equation*}
    Therefore, independently of the sign of $\dot D_p$, it holds
    \begin{equation}\label{ineq:diff dot D_p with J D_p Theta}
        \dot{D_p} \le J D_p \sqrt{\abs{\log D_p}\Theta\left(\abs{\log D_p}\right)} \,= J\sqrt{D_p \, \varphi_\Theta(D_p)}
    \end{equation}
    with $J := p \times [1 + C_{p,\Theta;d}A \times ( 1_{p=1} + 1_{1<p<+\infty}e^{1/p}\max_{s \in [0,T]}C_U(s))]$. While $C_U(t)$ depends implicitly on $D_p(t)$, $C_U(t) \in L^\infty([0,T])$ in the regime \eqref{ineq:strict regime} (which enforces $1 < (p'\abs{\log c_{p,\Theta;d}}')' < (p'r'(t))' \le p$ in the proof of \cref{prop:L^p estimate}), and taking its maximum on $[0, T]$ removes the dependence of $J(t)$ on $D_p(t)$; 
    \begin{equation*}
        \max_{s \in [0, T]} C_U(s) = 2e^{1/p'}C_d \max_{l \in \left[\left(p'\abs{\log c_{p,\theta;d}}'\right)', p\right]} C_{\HW;\, l}.
    \end{equation*}
    Note that this maximum is computable without knowing a priori the values of $D_p(t)$ on $[0, T]$. Now that the dependence has been removed, the differential inequality \eqref{ineq:diff dot D_p with J D_p Theta} can be closed via \cref{lem:Osgoog} under \cref{ass:Osgood};
    \begin{equation}\label{ineq:D_p(t) controlled by D_p(0)}
        D_p(t) \le \Psi_{p,\Theta;d}^{-1}\left(\Psi_{p,\Theta;d}(D_p(0)) - \int_0^t J(s) \: ds\right)
    \end{equation}
    as long as
    \begin{equation}\label{ineq:final regime with condition for D_p(0)}
        \Psi_{p,\Theta;d}(D_p(0)) \ge \int_0^t J(s) \: ds.
    \end{equation}
    It remains to translate \eqref{ineq:D_p(t) controlled by D_p(0)} and \eqref{ineq:final regime with condition for D_p(0)} in terms of $p$-Wasserstein distances. On one hand, $D_p(t)$ controls $W_p^p(f_1(t), f_2(t))$ recalling \eqref{ineq:control W_p by D_p}; indeed, since $(Z_1, Z_2)_{\#}\pi_0$ has marginals $f_1$ and $f_2$, and $\lambda \ge 1$ in the regime \eqref{ineq:strict regime},
    \begin{multline}\label{ineq:W_p^p controlled by D_p}
    W_p^p\left(f_1(t), f_2(t)\right) := \inf_{\gamma \in \Pi(f_1(t), f_2(t))} \int_{(\XRd)^2} \abs{x-y}^p + \abs{v-w}^p\: d\gamma(x,v,y,w) 
        \\ \le \int_{(\XRd)^2} \abs{x-y}^p + \abs{v-w}^p \: d\Big[\left(Z_1, Z_2\right)_\#\pi_0\Big](x,v,y,w) \le D_p(t).
    \end{multline}
    On the other hand, $D_p(0)$ is essentially controlled by $W_p^p(f_1(0), f_2(0))$; recalling the optimality of the coupling $\pi_0$ and the fact that $\lambda(0) \ge 1$ in the regime \eqref{ineq:strict regime}, we get
    \begin{equation*}
        D_p(0) \le \lambda(0)\int_{(\XRd)^2} \abs{x-y}^p + \abs{v-w}^p \:d\pi_0(x,v,y,w) := \lambda(0)W_p^p\left(f_1(0), f_2(0)\right),
    \end{equation*}
    which we rewrite as 
    \begin{equation*}
        \frac{D_p(0)}{\sqrt{\abs{\log D_p(0)}\Theta\left(\abs{\log D_p(0)}\right)}^p} \le W_p^p\left(f_1(0), f_2(0)\right),
    \end{equation*}
    and, denoting by $\Phi^{-1}_{p,\Theta}$ the inverse of the map $s \mapsto s/\sqrt{\abs{\log s}\Theta(\abs{\log s})}^{p}$ near the origin (note that this inverse exists since $\Theta$ is nondecreasing), then there is a constant $c_\Theta > 0$ such that for sufficiently small initial Wasserstein distance between the two solutions, say, $W_p^p(f_1(0), f_2(0)) \le c_\Theta$, it holds
    \begin{equation}\label{ineq:D_p(0) W_p^p(f_1(0), f_2(0)) control}
        D_p(0) \le \Phi^{-1}_{p,\Theta}\left[W_p^p\left(f_1(0), f_2(0)\right)\right].
    \end{equation}
    Combining \eqref{ineq:W_p^p controlled by D_p} and \eqref{ineq:D_p(0) W_p^p(f_1(0), f_2(0)) control} with \eqref{ineq:D_p(t) controlled by D_p(0)} gives the desired $W_p$ stability estimate \eqref{ineq:stability estimate W_p VP}, while combining these two with \eqref{ineq:final regime with condition for D_p(0)} provides the corresponding assumption \eqref{ineq:initial condition W_p^p} on $W_p^p(f_1(0), f_1(0))$. 
    
    To conclude, note that $J(t) \in L^1([0, T])$ since $A(t) \in L^1([0, T])$ by \cref{ass:L^1-bound}.
\end{proof}

\section{Beyond the Vlasov--Poisson system}\label{Sec:Beyond the Vlasov--Poisson system}
The kinetic Wasserstein distance approach is also well-suited for studying the magnetized Vlasov--Poisson system in dimension $d = 2$ or $d = 3$ on $\X$, namely,
\begin{equation}\label{sys:Vlasov--Poisson weak B}\tag{VPB}
    \begin{cases}
        \partial_t f + v\cdot \nabla_x f + \left(E_f + v \wedge B\right)\cdot \nabla_{v} f = 0, \\
        E_f = -\nabla_x U_f, \quad -\Delta_x U_f := \rho_f - 1_{\X = \Td},\\
        \rho_f := \int_{\Rd} f \: dv,
    \end{cases}
\end{equation}
where the external varying magnetic field $B$ is bounded in both time and space, and with $\log$-Lipschitz regularity in position. Similarly to \eqref{sys:Vlasov--Poisson}, we denote $B(X_i) := B(t; X_i)$ for $i=1,2$, and the system of characteristics is then given by
\begin{equation*}\label{sys:Characteristics VPB}\tag{CH-VPB}
    \dot{X} = V, \quad \dot{V} = E_f + V \wedge B, \quad X(0; x, v) = x, \quad V(0; x, v) = v.
\end{equation*}
The presence of an external magnetic field induces anisotropy between the two solutions, manifesting as the following additional term when analyzing the evolution of the quantity $D_p$:
\begin{equation}\label{eq:extra term}
    \int_{(\XRd)^2} \abs{V_1 - V_2}^{p-2}(V_1 - V_2) \cdot \left[V_2 \wedge \left(B(X_1) - B(X_2)\right)\right] \: d\pi_0,
\end{equation}
where, by orthogonality, we have rewritten
\begin{equation*}
    \left(V_1 - V_2\right)\cdot\left(V_1 \wedge B(X_1) - V_2 \wedge B(X_2)\right) = (V_1 - V_2) \cdot \left[ V_2 \wedge \left(B(X_1) - B(X_2)\right)\right].
\end{equation*}
This isolated flow $V_2$ can be controlled in terms of the electric and the external magnetic fields:
\begin{lemma}[{\cite[Lemma 2.7]{rege_stability_2024}}]\label{lem:V control E B}
    Let $f$ be a weak solution to \eqref{sys:Vlasov--Poisson weak B}. Then for all $t \in [0, T], v \in \Rd$, it holds
    \begin{equation*}\label{ineq:estimate velocity field E B}
        \norm{V(t;\cdot,v)}_{L^\infty(\X)} \le \abs{v}e^{t\norm{B}_{L^\infty([0, T] \times \X)}} + \int_0^t \norm{E(s)}_{L^\infty(\X)}e^{(t-s)\norm{B}_{L^\infty([0, T] \times \X)}} \: ds.
    \end{equation*}
\end{lemma}
Applying H\"older's inequality with respect to the measure $\pi_0$ to \eqref{eq:extra term} and \cref{lem:V control E B}, we obtain
\begin{multline}
    \int_{(\XRd)^2} \abs{V_1 - V_2}^{p-2}(V_1 - V_2) \cdot \left[V_2 \wedge \left(B(X_1) - B(X_2)\right)\right] \: d\pi_0 \\ \leq D_p^{\frac{1}{p'}}\left(\int_{(\XRd)^2} \abs{V_2}^p\abs{B(X_1) - B(X_2)}^p \: d\pi_0\right)^{\frac{1}{p}} \le D_p^{\frac{1}{p}}\left(T_3 + T_4\right), \label{ineq:control of momemtum isolated flow split}
\end{multline}
where
\begin{equation*}\label{eq:T_3}
    T_3(t) := 2 e^{t\norm{B}_{L^\infty([0, T)\times\X)}} \left(\int_{(\XRd)^2} \abs{w}^p\abs{B(t; X_1(t; x,v) - B(t; X_2(t;y,w))}^p \: d\pi_0(x,v,y,w)\right)^{\frac{1}{p}}
\end{equation*}
and
\begin{equation*}\label{eq:T_4}
    T_4(t) := 2\int_0^t \norm{E_{f_2}(s)}_{L^\infty(\X)}e^{(t-s)\norm{B}_{L^\infty([0, T)\times\X)}} \:ds \left(\int_{(\XRd)^2} \abs{B(X_1) - B(X_2)}^p \: d\pi_0 \right)^{\frac{1}{p}}.
\end{equation*}

We start by estimating $T_3$; we apply H\"older's inequality for some $r \ge 1$,
\begin{equation*}\label{ineq:T_3 intermediate ineq}
    T_3(t) \le e^{t\norm{B}_{L^\infty([0, T]\times\X)}} \left(\int_{(\XRd)^2}\abs{w}^{pr} \: d\pi_0(x,v,y,w) \right)^{\frac{1}{pr}}\left(\int_{(\XRd)^2} \abs{B(X_1) - B(X_2)}^{pr'} \: d\pi_0 \right)^{\frac{1}{pr'}}.
\end{equation*}
Given that the external magnetic field has a $\log$-Lipschitz regularity, we can control its H\"older norm. Indeed, there is a constant $C_B > 0$ such that for all $0 < \gamma < 1$ and $\abs{x-y} < 1/e$, it holds that
\begin{equation*}
    \abs{B(t;x) - B(t;y)} \le C_B \frac{e^{-1}}{1-\gamma}\abs{x - y}^\gamma, \quad \sup_{\abs{x-y} < 1/e} \abs{x-y}^{1-\gamma} \abs{\log \abs{x-y}} = \frac{e^{-1}}{1-\gamma}.
\end{equation*}
Then, as long as $1/r' \ge \gamma$,
\begin{align*}
    \left(\int_{(\XRd)^2} \abs{B(X_1) - B(X_2)}^{pr'} \: d\pi_0 \right)^{\frac{1}{pr'}} &\le C_B\frac{e^{-1}}{1-\gamma}\left(\int_{\abs{X_1 - X_2} < 1/e} \abs{X_1 - X_2}^p \: d\pi_0 \right)^{\frac{\gamma}{p}} \\ &\quad + \norm{B}_{L^\infty([0,T]\times\X)}e^{\frac{1}{r'}}\left(\int_{\abs{X_1 - X_1} \ge 1/e} \abs{X_1 - X_2}^p \: d\pi_0\right)^{\frac{1}{pr'}}.
\end{align*}
We choose $r = \abs{\log(D_p/\lambda)}$ so that $\gamma := 1/r' = 1 - 1/\abs{\log(D_p/\lambda)}$, and hence 
\begin{equation*}
    \left(\frac{D_p}{\lambda}\right)^\gamma = \left(\frac{D_p}{\lambda}\right)^{\frac{1}{r'}} = \left(\frac{D_p}{\lambda}\right)\left(\frac{D_p}{\lambda}\right)^{\frac{1}{\log(D_p/\lambda)}} \le e\left(\frac{D_p}{\lambda}\right), \quad \frac{1}{1-\gamma} = \abs{\log\left(\frac{D_p}{\lambda}\right)}.
\end{equation*}
Note that
\begin{equation*}
    \left(\int_{(\XRd)^2}\abs{w}^{pr} \: d\pi_0(x,v,y,w) \right)^{\frac{1}{pr}} = \left(\int_{\XRd} \abs{w}^{pr} \:df_2(0; y, w)\right)^{\frac{1}{pr}} \le \overline{\Theta}(pr)\norm{\abs{w}}_{Y^{\overline{\Theta}}(\XRd; df_2)},
\end{equation*}
for some growth function $\overline{\Theta}$, which naturally leads us to impose a bound on the moments of one of the two solutions.
\begin{assumption}\label{ass:moments}
    The moment of one of the two Lagrangian weak solutions belongs initially to a weighted Yudovich space;
    \begin{equation*}
        \norm{\abs{w}}_{Y^{\overline{\Theta}}(\XRd;\, df_2(0))} := \sup_{1 \le r < +\infty} \frac{\left(\int_{\XRd} \abs{w}^r f_2(0; y, w) \: dydw\right)^{\frac{1}{r}}}{\overline{\Theta}(r)} < +\infty.
    \end{equation*}
\end{assumption}
By controlling the extra magnetized terms as above, the second author \cite[Theorem 2.5]{rege_propagation_2023} obtained a stability criterion for \eqref{sys:Vlasov--Poisson weak B}. The main contribution was showing that the stability of solutions holds under an extra assumption, of type \cref{ass:moments}, on the growth of the velocity moments of one of the initial distributions, and crucially this assumption is verified by Maxwellians. Combining these ideas with the analysis from \cref{sec:Proof of the main result}, we proceed to find a stability estimate for \eqref{sys:Vlasov--Poisson weak B}. We set $\widetilde\Theta(\cdot) \coloneqq \overline{\Theta}(p \times \cdot)$ and assume that the growth function dominates that of an exponential Orlicz space; specifically, $\widetilde{\Theta}(l) \ge l^{1/\alpha}$ for some $1 \le \alpha < +\infty$, which is a natural growth function ensuring that the magnetized terms are not controlled by the electric ones in the bounded setting (see \cite[Remark 2.4]{rege_stability_2024}). In the Yudovich setting, to prevent this from happening, we may without loss of generality assume that $\widetilde\Theta \ge \Theta$. Now, we set $\lambda := \sqrt{\abs{\log D_p}\widetilde\Theta(\abs{\log D_p})}^p$, and thanks to both \eqref{ineq:Theta kinetic useful} under \cref{ass:Osgood} and \cref{ass:Theta-two-inequalities-for-lambda} for the growth function $\widetilde{\Theta}$ in the new regime $D_p(t) \le \min\{c_{p,\widetilde\Theta;d}, c_{p,\Theta,d}\}$, we can estimate $T_3$ as follows:
\begin{align*}
    T_3(t) &\le e^{t\norm{B}_{L^\infty([0, T]\times\X)}} \norm{\abs{w}}_{Y^{\widetilde\Theta}(\XRd; df_2)}\left(\frac{D_p}{\lambda}\right)^{\frac{1}{p}}\widetilde{\Theta}\left(\abs{\log \left(\frac{D_p}{\lambda}\right)}\right) \\
        &\quad \times \left(e^{-\frac{1}{p'}}C_B \abs{\log \left(\frac{D_p}{\lambda}\right)} + e\norm{B}_{L^\infty([0, T]\times\X)}\right) \\
        &\le C_{B,\widetilde\Theta,p;d} e^{t\norm{B}_{L^\infty([0, T]\times\X)}} \norm{\abs{w}}_{Y^{\widetilde\Theta}(\XRd; df_2)} \times D_p^{\frac{1}{p}}\sqrt{\abs{\log D_p}\widetilde{\Theta}\left(\abs{\log D_p}\right)}\numberthis\label{ineq:T_3 final estimate}
\end{align*}
with $C_{B,\widetilde{\Theta},p;d} := \overline{C}_{p,\widetilde{\Theta};d}(e^{-1/p'}C_B C_{p,\widetilde{\Theta};d} + e\norm{B}_{L^\infty([0,T]\times\X)})$.

We proceed in a similar fashion to estimate $T_4$, but without having to rely on \cref{ass:moments}, and we obtain
\begin{equation}\label{ineq:T_4 final estimate}
    T_4(t) \le C_{B,\widetilde{\Theta},p;d}\int_0^t \left(1 + \norm{\rho_{f_2}(s)}_{Y^\Theta(\X)}\right)e^{(t-s)\norm{B}_{L^\infty([0, T)\times\X)}} \:ds \times  D_p^{\frac{1}{p}}\sqrt{\abs{\log D_p}}.
\end{equation}

Without loss of generality, we may assume that $\widetilde\Theta \ge \Theta$. Under the \cref{ass:L^1-bound,ass:varphi_theta-continuous,ass:Theta-non-drecreasing-and-concave} for the growth function $\Theta$, \cref{ass:Theta-two-inequalities-for-lambda,ass:Osgood} for the growth function $\widetilde\Theta \ge \Theta$, and \cref{ass:moments} on the moments of one of the two solutions, we thus obtain, combining \eqref{ineq:control of momemtum isolated flow split}, \eqref{ineq:T_3 final estimate} and \eqref{ineq:T_4 final estimate} to the analysis of \cref{sec:Proof of the main result}, the following $W_p$ stability estimate for \eqref{sys:Vlasov--Poisson weak B}:
\begin{equation}\label{ineq:stability estimate W_p VPB}
    W_p^p\left(f_1(t), f_2(t)\right) \le \Psi_{p,\widetilde\Theta;d}^{-1}\left(\Psi_{p,\widetilde\Theta;d}\left(\Phi^{-1}_{p,\widetilde\Theta} \left[W_p^p\left(f_1(0), f_2(0)\right)\right]\right) - \int_0^t \widetilde{J}(s) \: ds\right)
\end{equation}
with
\begin{multline*}
    \widetilde{J}(t) := p \times \left[1 + \widetilde{C}_{p,\widetilde\Theta;d}A(t) \times \left( 1_{p=1} + 1_{1<p<+\infty}e^{1/p}\max_{s \in [0,T]}C_U(s)\right)\right. \\
    \left.+\, C_{B,\widetilde\Theta,p;d} \left (e^{t\norm{B}_{L^\infty([0, T]\times\X)}} \norm{\abs{w}}_{Y^{\widetilde\Theta}(\XRd;\, df_2(0))} + \int_0^t \left(1 + \norm{\rho_{f_2}(s)}_{Y^\Theta(\X)}\right)e^{(t-s)\norm{B}_{L^\infty([0, T)\times\X)}} \:ds\right)\right].
\end{multline*}

While the stability estimate \eqref{ineq:stability estimate W_p VPB} provides a uniqueness criterion for \eqref{sys:Vlasov--Poisson weak B} as a corollary, it remains necessary to construct weak solutions for which the Lebesgue norms of the macroscopic density grow as fast as $\Theta$. Crippa, Inversi, Saffirio, and Stefani \cite[Theorem 1.7 \& Proposition 1.8]{crippa_existence_2024} constructed such Lagrangian weak solutions for \eqref{sys:Vlasov--Poisson} with Yudovich macroscopic density building upon the construction of Miot \cite[Theorems 1.2 and 1.3]{miot_uniqueness_2016}, which was generalized to \eqref{sys:Vlasov--Poisson weak B} by the second author \cite[Theorem 2.1]{rege_propagation_2023}, and, incidentally, 
the same construction can be directly adapted to \eqref{sys:Vlasov--Poisson weak B}, 
yielding a solution to \eqref{sys:Vlasov--Poisson weak B} for $\X = \R^3$ starting from the initial datum
\begin{equation*}
    f_0(x,v) = \frac{1_{(-\infty,0]} \left(\abs{v}^2 - \theta(x)^{\frac{2}{3}}\right)}{\tfrac{4}{3}\pi\norm{\theta}_{L^1(\R^3)}},
\end{equation*}
where the nonnegative function $\theta \in Y^\Theta(\R^3)$ is such that
\begin{equation*}
    \int_{\R^3} \max\left\{ 1, \abs{x}^p\right\} \theta(x) \: dx <+\infty.
\end{equation*}

\noindent \textbf{Acknowledgments:} J. Junn\'e acknowledges
financial support from the Dutch Research Council (NWO): This publication is part of the project Interacting particle systems and Riemannian geometry (with project number OCENW.M20.251) of the research program Open Competitie ENW which is (partly) financed by the Dutch Research Council (NWO) \href{https://www.nwo.nl/en/projects/ocenwm20251}{\includegraphics[width=\fontcharht\font`\B+\fontcharht\font`\s, height=\fontcharht\font`\B+\fontcharht\font`\s]{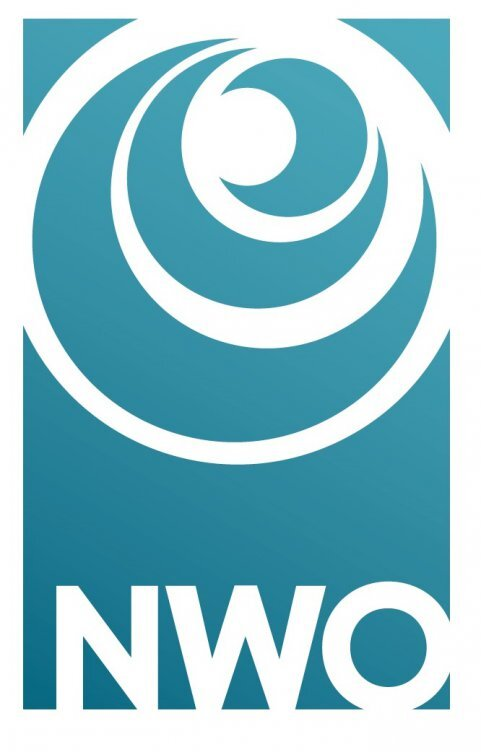}}.

A. Rege was supported by the NCCR SwissMAP which was funded by the Swiss National
Science Foundation grant number 205607. A. Rege would like to thank the Swiss National Science Foundation for its financial support.

\printbibliography
\end{document}